\documentclass[12pt,a4paper]{amsart} 

\usepackage{latexsym}
\usepackage{amsmath}
\usepackage{amsthm}
\usepackage{tikz}
\usepackage{latexsym}
\usepackage{amssymb}
\usepackage[a4paper , lmargin = {2cm} , rmargin = {2cm} , tmargin = {2.5cm} , bmargin = {2.5cm} ]{geometry} 
\usepackage{caption}
\usepackage{hyperref}

\DeclareMathOperator{\GL}{GL}

\renewcommand{\phi}{\varphi}
\renewcommand{\emptyset}{\varnothing}

\newcommand{\Aut}{{\rm Aut}}  
\newcommand{\Inn}{{\rm Inn}} 
\newcommand{\PC}{{\rm Pc}}
\newcommand{\Z}{\ensuremath{\mathbb{Z}}}

\makeatletter
\def\theenumi{\@roman\c@enumi}
\makeatother

\theoremstyle{plain}
\newtheorem*{PropositionA}{Proposition A}
\newtheorem*{TheoremB}{Theorem B}
\newtheorem*{CorollaryC}{Corollary C}
\newtheorem*{CorollaryD}{Corollary D}
\newtheorem*{PropositionE}{Proposition E}
\newtheorem*{CorollaryF}{Corollary F}
\newtheorem*{TheoremG}{Theorem G}

\newtheorem*{CorollaryH}{Corollary H}

\newtheorem*{CorollaryI}{Corollary I}
\newtheorem*{CorollaryJ}{Corollary J}
\newtheorem*{Question}{Question}

\newtheorem{theorem}[subsection]{Theorem}
\newtheorem{conjecture}[subsection]{Conjecture}
\newtheorem{lemma}[subsection]{Lemma}

\newtheorem{corollary}[subsection]{Corollary}

\newtheorem{proposition}[subsection]{Proposition}
\theoremstyle{definition}
\newtheorem{definition}[subsection]{Definition}
\newtheorem{question}[subsection]{Question}

\newtheorem{remark}[subsection]{Remark}

\title[]{On normal subgroups in automorphism groups}
\author{Philip M\"oller and Olga Varghese}
\date{\today}

\address{Philip M\"oller\\
	Department of Mathematics\\
	University of M\"unster\\ 
	Einsteinstra\ss e 62\\
	48149 M\"unster, Germany.}
\email{philip.moeller@uni-muenster.de}

\address{Olga Varghese\\ Institute of Mathematics, Heinrich-Heine-University Düsseldorf, Universitätsstra{\upshape{\ss}}e 1, 40225, Düsseldorf, Germany.}
\email{olga.varghese@hhu.de}

\begin{document}
\begin{abstract} 
We describe the structure of virtually solvable normal subgroups in the automorphism group of a right-angled Artin group ${\rm Aut}(A_\Gamma)$. In particular,  we prove that a finite normal subgroup in ${\rm Aut}(A_\Gamma)$ has at most order two and if $\Gamma$ is not a clique, then any finite normal subgroup in ${\rm Aut}(A_\Gamma)$ is trivial. This property has implications to automatic continuity and to $C^\ast$-algebras: every algebraic epimorphism $\varphi\colon L\twoheadrightarrow{\rm Aut}(A_\Gamma)$ from a locally compact Hausdorff group $L$ is continuous if and only if $A_\Gamma$ is not isomorphic to $\mathbb{Z}^n$ for any $n\geq 1$. Further, if $\Gamma$ is not a join and contains at least two vertices, then the set of invertible elements is dense in the reduced group $C^\ast$-algebra of $\Aut(A_\Gamma)$. We obtain similar results for ${\rm Aut}(G_\Gamma)$ where $G_\Gamma$ is a graph product of cyclic groups. Moreover, we give a description of the center of $\Aut(G_\Gamma)$ in terms of the defining graph $\Gamma$.

\vspace{1cm}
\hspace{-0.6cm}
\medskip

\medskip
\hspace{-0.5cm}{\bf 2010 Mathematics Subject Classification.} Primary: 20E36; Secondary: 20E06.

\thanks{The first author is funded
	by a stipend of the Studienstiftung des deutschen Volkes and  by the Deutsche Forschungsgemeinschaft (DFG, German Research Foundation) under Germany's Excellence Strategy EXC 2044--390685587, Mathematics M\"unster: Dynamics-Geometry-Structure. The second author is supported by DFG grant VA 1397/2-2. This work is part of the PhD project of the first author.}
\end{abstract}

\maketitle
\section{Introduction}
In the study of certain questions in mathematics, there are often times two extreme cases that come to mind instantly. For example the subgroups of a certain group can either be ``small'' (e.g. finite/virtually abelian/virtually solvable) or ``large'' (e.g. contain a non-abelian free group/have a subgroup of finite index with a non-abelian free quotient) or anything in between. Sometimes, when the object in question is particularly well-behaved, only the two extreme cases can happen. In group theory, the above dichotomy is called the  
Tits alternative. This name is derived from the result of Tits who showed that a finitely generated linear group is virtually solvable or contains a non-abelian free group as a subgroup \cite{Tits}.

\subsection{Normal subgroups in automorphism groups of right-angled Artin groups} 
Following the notion of \cite{Bridson} a group $G$ is called \emph{full-sized} if it contains a non-abelian free group $F_2\hookrightarrow G$. Here we are interested in the structure of non full-sized normal subgroups in the automorphism group of a group that is defined via a graph. Given a finite simplicial graph $\Gamma$ with  the vertex set $V(\Gamma)$ and the edge set $E(\Gamma)$, the \emph{right-angled Artin group} $A_\Gamma$ is defined as follows
$$A_\Gamma:=\langle V(\Gamma)\mid vw=wv \text{ whenever } \{v,w\} \in E(\Gamma) \rangle.$$ For instance, if $\Gamma$ is a clique, then $A_\Gamma$ is the free abelian group $\mathbb{Z}^{|V(\Gamma)|}$ and if $\Gamma$ has no edges, then $A_\Gamma$ is the free group of rank $|V(\Gamma)|$.

The starting point for our investigations on normal subgroups in $\Aut(A_\Gamma)$ is the well-studied group $\Aut(\Z^n)\cong {\rm GL}_n(\mathbb{Z})$. Given a normal subgroup $N\trianglelefteq{\rm GL}_n(\mathbb{Z})$, it is known that  $N$ is finite or $N$ is full-sized. Since ${\rm GL}_n(\mathbb{Z})$ shares many properties with the automorphism group of a free group and to some extend also with the automorphism group of an arbitrary right-angled Artin group, our goal here is to understand the structure of non full-sized normal subgroups in $\Aut(A_\Gamma)$. It was proven by Horbez in \cite{Horbez} that a non full-sized subgroup in $\Aut(A_\Gamma)$ is virtually solvable. Thus we are in particular interested in the structure of virtually solvable normal subgroups in $\Aut(A_\Gamma)$. We start by considering normal subgroups in the automorphism group of a centerless right-angled Artin group.
\begin{PropositionA}
	Let $\Aut(A_\Gamma)$ be the automorphism group of a centerless right-angled Artin group $A_\Gamma$ and let $N\trianglelefteq\Aut(A_\Gamma)$ be a normal subgroup. If $N$ is non-trivial, then $N$ is full-sized.
\end{PropositionA}	

Given a right-angled Artin group $A_\Gamma$, then the center of $A_\Gamma$ is a direct factor of $A_\Gamma$ and is itself a right-angled Artin group given by the induced subgraph $\Delta$ whose vertices $v$ are precisely the ones satisfying $st(v)=V(\Gamma)$ where $st(v):=\left\{w\in V(\Gamma)\mid \left\{v,w\right\}\in E(\Gamma)\right\}\cup\left\{v\right\}$. Hence, many right-angled Artin groups have non-trivial centers. We  denote the vertices in $V(\Delta)$ by $\left\{v_1,\ldots, v_n\right\}$ and  the vertices in $V(\Gamma)-V(\Delta)$ by $\left\{w_1,\ldots, w_m\right\}$.

Let us discuss the center and some other properties of the right-angled Artin group defined via the graph $\Gamma$ in Figure 1.

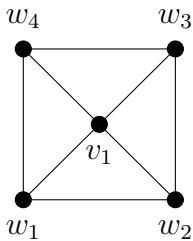
\begin{figure}[h]
	\begin{center}
	\captionsetup{justification=centering}
		\begin{tikzpicture}
			\draw[fill=black]  (0,0) circle (3pt);
			\node at (0,-0.4) {$w_1$};
			\draw[fill=black]  (2,0) circle (3pt);
			\node at (2,-0.4) {$w_2$};	
			\draw[fill=black]  (2,2) circle (3pt);
			\node at (2,2.4) {$w_3$};
			\draw[fill=black]  (0,2) circle (3pt);
			\node at (0,2.4) {$w_4$};	
			\draw[fill=black]  (1,1) circle (3pt);
			\node at (1,0.6) {$v_1$};		
			\draw (0,0)--(2,0);	
			\draw (2,0)--(2,2);	
			\draw (2,2)--(0,2);
			\draw (0,0)--(0,2);
			\draw (0,0)--(1,1);	
			\draw (2,0)--(1,1);	
			\draw (2,2)--(1,1);
			\draw (0,2)--(1,1);
		\end{tikzpicture}

\caption{Graph $\Gamma$.}

\end{center}
\end{figure}

The center of $A_\Gamma$ is generated by $v_1$ and is a direct factor of $A_\Gamma$, thus $A_\Gamma\cong \langle v_1\rangle\times\langle w_1, w_2, w_3, w_4\rangle$. For $w_i\in V(\Gamma)-V(\Delta)$ we define an automorphism $\rho_{w_iv_1}\in \Aut(A_\Gamma)$ as follows: $\rho_{w_iv_1}(w_i)=w_iv_1$, $\rho_{w_iv_1}(w_j)=w_j$, $i\neq j$ and $\rho_{w_iv_1}(v_1)=v_1$. One can show that the subgroup $\langle \rho_{w_iv_1}\mid i=1,\ldots,4\rangle\subseteq \Aut(A_\Gamma)$ is normal and it is isomorphic to $\mathbb{Z}^4$. This normal subgroup is not maximal under all non full-sized normal subgroups in $\Aut(A_\Gamma)$ since the subgroup $\langle \rho_{w_iv_1}, \iota\mid i=1,\ldots,4\rangle$ where $\iota(v_1)=v_1^{-1}$ and $\iota(w_i)=w_i$ for $i=1,\ldots, 4$ is also normal. Furthermore, this normal subgroup is isomorphic to $\Z^4\rtimes \Z/2\Z$ and it is indeed maximal under all non full-sized normal subgroups in $\Aut(A_\Gamma)$. We show that a maximal non full-sized normal subgroup in the automorphism group of a general right-angled Artin group is unique.   
\begin{TheoremB}	 
	Let $\Aut(A_\Gamma)$ be the automorphism group of a right-angled Artin group $A_\Gamma$ and let $N\trianglelefteq{\rm Aut}(A_\Gamma)$ be a non-trivial maximal (with respect to inclusion) normal subgroup. 
	
	If $N$ is non full-sized, then either
	\begin{enumerate}
	\item $V(\Gamma)-V(\Delta)=\emptyset$ and $N=\left\{id, \iota\right\}$ where $\iota(v_i)=v_i^{-1}$ for all $v_i\in V(\Delta)$
	\end{enumerate}
or
\begin{enumerate}
	\item[(ii)] $V(\Delta)=\left\{v_1,...,v_n\right\}$, $V(\Gamma)-V(\Delta)=\left\{w_1,...,w_m\right\}$ for $n,m\geq 1$ and $N\cong T_\Delta\rtimes \left\{id, \iota\right\}\cong \Z^{n\cdot m}\rtimes\Z/2\Z$ where $$T_\Delta:=\left\{f\in\Aut(A_\Gamma)\mid w_i^{-1}f(w_i)\in A_\Delta\text{ and } f(v_j)=v_j\text{ for all } i=1,\ldots, m\text{ and } j=1,\ldots, n\right\}.$$ 
	\end{enumerate}
 Furthermore, let $N\trianglelefteq\Aut(A_\Gamma)$ be a non-trivial minimal non full-sized normal subgroup and set $n:=|V(\Delta)|$ and $m:=|V(\Gamma)-V(\Delta)|$. If $n,m\geq 1$, then $N\cong \Z^l$ for some $n\leq l\leq n\cdot m$.
\end{TheoremB}

As an immediate corollary we have
\begin{CorollaryC}
	The automorphism group of a right-angled Artin group $\Aut(A_\Gamma)$ does have non-trivial finite normal subgroups if and only if $\Gamma$ is a clique.
\end{CorollaryC}

It is a known fact that a centerless group has a centerless automorphism group. The center of a right-angled Artin group can be easily read off the defining graph $\Gamma$ as discussed above. Thus, there are many right-angled Artin groups with non-trivial center. Nevertheless, we show that the automorphism group of a right-angled Artin group is in most cases centerless.
\begin{CorollaryD}
	The automorphism group of a right-angled Artin group $\Aut(A_\Gamma)$ is centerless if and only if $\Gamma$ is not a clique.
\end{CorollaryD}

The center of the automorphism group of a right-angled Artin group was also determined in \cite{Fullarton}, but the proof of \cite[Prop. 5.1]{Fullarton} contains a small gap.

\subsection*{Structure of the proofs of Proposition A and Theorem B}
The crucial ingredient in the proof of Proposition A is a result by Baudisch which says that a two generated subgroup in $A_\Gamma$ is abelian or isomorphic to $F_2$ \cite[Thm. 1.2]{Baudisch}. This result is used to prove that a normal subgroup in $A_\Gamma$ is contained in the center or is full-sized. Using the subgroup consisting of inner automorphisms of $A_\Gamma$ and the characterization of non full-sized normal subgroups in $A_\Gamma$ we are able to prove Proposition A. For the proof of Theorem B we use Proposition A and a short exact sequence which is given by the fact that $A_\Delta$ is a characteristic subgroup:
$\left\{id\right\}\to T_\Delta\to \Aut(A_\Gamma)\to\Aut(A_\Delta)\times\Aut(A_{\Gamma-\Delta})\to\left\{id\right\}$.
\vspace{0.3cm}

\subsection{Normal subgroups in automorphism groups of graph products}
Right-angled Artin groups are closely related to graph products of groups. Given a finite simplicial graph $\Gamma$ and a collection of non-trivial groups $\{ G_u \mid u \in V(\Gamma) \}$ indexed by the vertex-set $V(\Gamma)$ of $\Gamma$, the \emph{graph product} $G_\Gamma$ is defined as the quotient
$$ \left( \underset{u \in V(\Gamma)}{\ast} G_u \right) / \langle \langle [g,h]=1, \ g \in G_u, h \in G_v, \{u,v\} \in E(\Gamma) \rangle \rangle$$
where $E(\Gamma)$ denotes the edge-set of $\Gamma$. If all vertex groups are infinite cyclic, then $G_\Gamma$ is a right-angled Artin group and if all vertex groups are cyclic of order two, then $G_\Gamma$ is called a \emph{right-angled Coxeter group} and will be denoted by $W_\Gamma$.

A natural question is if the results of Proposition A and Theorem B also hold for graph products of cyclic groups. Let us focus on the result of Proposition A that tells us that a non-trivial normal subgroup in the automorphism group of a centerless right-angled Artin group is always full-sized. Let us consider the centerless right-angled Coxeter group $\Z/2\Z*\Z/2\Z$ and its automorphism group. The group of inner automorphisms in $\Aut(\Z/2\Z*\Z/2\Z)$ is isomorphic to $\Z/2\Z*\Z/2\Z$.  Hence, the group $\Aut(\Z/2\Z*\Z/2\Z)$  has a normal subgroup isomorphic to $\Z/2\Z*\Z/2\Z$ and this group is obviously non full-sized.  
In the centerless case this is however the only way for non-trivial non full-sized normal subgroups to exist. 

\begin{PropositionE}(see Propositions \ref{CenterlessAutG} and \ref{FiniteVertexGroups})
Let $G_\Gamma$ be a graph product of non-trivial groups. We decompose $G_\Gamma\cong G_{\Gamma_1}\times\ldots\times G_{\Gamma_n}$ as a direct product where the factors are directly indecomposable special parabolic subgroups. Assume that $G_{\Gamma_i}\ncong\Z/2\Z*\Z/2\Z$ for all $i=1,\ldots, n$.
\begin{enumerate}
\item If $|V(\Gamma_i)|\geq 2$ for $i=1,\ldots,n$, then a non-trivial normal subgroup in $\Aut(G_\Gamma)$ is full-sized.
In particular, if $G_\Gamma$ is a centerless graph product of abelian groups, then a non-trivial normal subgroup in $\Aut(G_\Gamma)$ is full-sized. 
\item If all vertex groups are finite, then a normal subgroup in $\Aut(G_\Gamma)$ is finite or full-sized.
\end{enumerate}
\end{PropositionE}

In the special case where the graph product is a right-angled Coxeter group we prove that if one of the factors in the direct decomposition of  $W_\Gamma$ is isomorphic to the infinite dihedral group, then $\Aut(W_\Gamma)$ has an infinite virtually abelian normal subgroup. More precisely, we define $J:=\left\{j\in\left\{1,\ldots,n\right\}\mid W_{\Gamma_j}\cong \Z/2\Z*\Z/2\Z\right\}$ and we show that $\prod_{j\in J}W_{\Gamma_j}$ as a subgroup of ${\rm Inn}(W_\Gamma)$ is normal in $\Aut(W_\Gamma)$. Thus the infinite dihedral group is a poison subgroup when it comes to the non-existence of infinite non full-sized normal subgroups in $\Aut(W_\Gamma)$ as the following corollary shows.

\begin{CorollaryF}
	Let $W_\Gamma$ be a right-angled Coxeter group. We decompose $W_\Gamma\cong W_{\Gamma_1}\times\ldots\times W_{\Gamma_n}$ as a direct product where the factors are directly indecomposable special parabolic subgroups. 
	
	There exists an infinite virtually abelian normal subgroup in $\Aut(W_\Gamma)$ if and only if there exists an $i\in\{1,\ldots,n\}$ such that $W_{\Gamma_i}\cong \Z/2\Z*\Z/2\Z$. Else every normal subgroup in $\Aut(W_\Gamma)$ is either finite or full-sized.
\end{CorollaryF}

Now, for a graph product of arbitrary cyclic groups we are interested in conditions on $\Gamma$ such that $\Aut(G_\Gamma)$ does not have non-trivial {\it finite} normal subgroups. For this it is easier to ``blow up'' the graph product a bit, by replacing the vertex with associated vertex group group $\Z/n\Z$, where $n=p_1^{m_1}\cdot...\cdot p_k^{m_k}$ by a complete graph on $k$ vertices with vertex groups isomorphic to $\Z/(p_i^{m_i}\Z)$ for $1\leq i\leq k$. This yields isomorphic graph products, but allows us to assume that the vertex groups are infinite cyclic or have prime power order.

We prove:

\begin{TheoremG} Let $G_\Gamma$ be a graph product of cyclic groups where the finite vertex groups have prime power orders. Let $\Delta$ be the induced subgraph of $\Gamma$ generated by the vertices $v\in V(\Gamma)$ such that $st(v)=V(\Gamma)$.

The automorphism group $\Aut(G_\Gamma)$ does not have non-trivial finite normal subgroup if and only if one of the following conditions holds:
\begin{enumerate}
    \item $V(\Delta)=\emptyset$ or
    \item $G_\Gamma\cong \Z/2\Z$ or
    \item $V(\Delta)=\{v_1\}$, $G_{v_1}\cong \Z/2\Z$, $V(\Gamma)-V(\Delta)=\{w_1,....,w_m\}$, $m\geq 1$ and $ord(G_{w_j})<\infty$ and $2\nmid ord(G_{w_j})$ for all $1\leq j\leq m$ or
    \item for every $v\in V(\Delta)$ we have $G_v\cong \Z$ and there exists a $w\in V(\Gamma)-V(\Delta)$ such that $G_w\cong \Z$.
\end{enumerate}

\end{TheoremG}

\subsection*{Structure of the proofs of Proposition E and Theorem G}
First, we give a description of non full-sized normal subgroups of a graph product of arbitrary vertex groups, see Proposition \ref{NormalGraphProduct} and Corollary \ref{ProductNormal}. Using several characteristic subgroups of $G_\Gamma$ and the induced group homomorphisms  we prove Proposition E and use this result to show Theorem G. 
\vspace{0.2cm}

\subsection{The center of the automorphism group of a graph product}
Similar methods as used for the proof of Theorem G allow us to precisely compute the center of the automorphism group of a graph product of cyclic groups in many situations. Since the center is heavily dependant on the structure of the graph one has to distinguish between many cases. First, we investigate the center of the automorphism group of a finitely generated abelian group and collect known and new results in Propositions \ref{centerFreeabelianAndFinite} and \ref{centerDelta}. A description of the center of $\Aut(G_\Gamma)$ where $\Gamma$ is not a clique is more complicated, see Proposition \ref{CenterDelta1}. It is known that the center of a Coxeter group is a finite abelian $2$-group. We prove that the center of the automorphism group of a right-angled Coxeter group is in most cases trivial and if the center is non-trivial, then it is also a finite abelian $2$-group.

\begin{CorollaryH}(see Corollary \ref{CenterRACG})
	Let $W_\Gamma$ be a right-angled Coxeter group. Let $\Delta$ be the induced subgraph of $\Gamma$ generated by the vertices $v\in V(\Gamma)$ such that $st(v)=V(\Gamma)$. Further, we denote the vertices of $\Delta$ by $V(\Delta)=\left\{v_1,\ldots, v_n\right\}$ and  $V(\Gamma)-V(\Delta)=\left\{w_1,\ldots, w_m\right\}$.  
	 
	 The center $Z(\Aut(W_\Gamma))$ is non-trivial if and only if $n=1$, $m\geq 1$ and there exists a vertex $w_j\in V(\Gamma)-V(\Delta)$ such that $st(w_j)\nsubseteq st(w_i)$ and $st(w_i)\nsubseteq st(w_j)$ for all $i\in\left\{1,\ldots, m\right\}, i\neq j$. Moreover, let $\Omega$ be a subset of $V(\Gamma)-V(\Delta)$ defined as follows: 
	$$\Omega:=\left\{w_j\mid st(w_j)\nsubseteq st(w_i)\text{ and } st(w_i)\nsubseteq st(w_j) \text{ for all }i\in\left\{1,\ldots, m\right\}, i\neq j\right\}.$$ 
	Then $\Omega$ is preserved under the action of ${\rm Isom}(\Gamma-\Delta)$ and $Z(\Aut(W_\Gamma))\cong (\Z/2\Z)^l$ where $l$ is equal to the cardinality of the set of orbits under this action.
	
\end{CorollaryH}

\subsection{Automatic continuity}
The results of Theorem G have an application to automatic continuity.
The aim of the concept of automatic continuity is to bring together two areas of mathematics: locally compact Hausdorff groups and discrete groups. The main question for us in this direction is:
\begin{Question}
Is every algebraic epimorphism $\varphi\colon L\twoheadrightarrow \Aut(G_\Gamma)$ from a locally compact Hausdorff group $L$ to the discrete group $\Aut(G_\Gamma)$ continuous?
\end{Question}
A direct consequence of \cite[Thm. B]{KramerVarghese} is that any algebraic epimorphism $\varphi\colon L\twoheadrightarrow G_\Gamma$ from a locally compact Hausdorff group $L$ into a graph product of cyclic groups $G_\Gamma$ is continuous if the diameter of $\Gamma$ is at least $3$. Using \cite[Thm. B]{KeppelerMoellerVarghese} we obtain a similar result where the target group is the automorphism group of a graph product of cyclic groups.

\begin{CorollaryI}(see Corollaries \ref{automaticAutG} and  \ref{automaticRAAG})
	Let $\varphi\colon L\twoheadrightarrow{\rm Aut}(G_\Gamma)$ be an algebraic epimorphism from a locally compact Hausdorff group $L$ to the automorphism group of a graph product of cyclic groups $G_\Gamma$. If $\Aut(G_\Gamma)$ does not have non-trivial finite normal subgroups, then $\varphi$ is continuous.
	
	In particular, if $V(\Delta)=\emptyset$ or  $V(\Delta)$ and $V(\Gamma)-V(\Delta)$ are both non-empty and all vertex groups in $V(\Delta)$ are infinite cyclic and there exists an infinite vertex group in $V(\Gamma)-V(\Delta)$, then $\varphi$ is continuous.
	
	Moreover, any algebraic epimorphism from a locally compact Hausdorff group to $\Aut(A_\Gamma)$ is continuous if and only if $\Gamma$ is not a clique.
\end{CorollaryI}

\subsection{The stable rank of the reduced group C*-algebra of the automorphism group of a graph product}
Associated to a group $G$, there is an associated reduced group $C^\ast$-algebra $C^\ast_r(G)$. We say that $C^\ast_r(G)$ has \textit{stable rank} $1$ if the set of invertible elements is dense in $C^\ast_r(G)$. Additional information about this in the context of acylindrically hyperbolic groups can be found in \cite{GerasimovaOsin}.
Proposition E and Theorem G have an application to $C^\ast_r(\Aut(G_\Gamma))$-algebras, more precisely using work of Gerasimova-Osin \cite{GerasimovaOsin} and Genevois \cite{Genevois} we obtain
\begin{CorollaryJ}
Let $G_\Gamma$ be a graph product of finitely generated groups. 
If $\Gamma$ is not a join, contains at least two vertices and $G_\Gamma\ncong\Z/2\Z*\Z/2\Z$, then the stable rank of the reduced group $C^\ast$-algebra of $\Aut(G_\Gamma)$ is equal to $1$. 

In particular, if $G_\Gamma$ is a right-angled Artin group and $\Gamma$ is not a join and contains at least two vertices, then the stable rank of the reduced group $C^\ast$-algebra of $\Aut(G_\Gamma)$ is equal to $1$.
\end{CorollaryJ}

\subsection*{Structure of the paper} In Chapter 2 we review some of the standard facts on graph products of groups $G_\Gamma$. It is natural to try to relate combinatorial properties of the defining graph $\Gamma$ and algebraic properties of the graph product $G_\Gamma$. We discuss these relations focusing on the question when a special subgroup of a graph product can be normal and when a graph product of groups can have a non-trivial non full-sized normal subgroup. These results allow us to prove Proposition A and Proposition E(i). 

Chapter 3 contains a brief summary of characteristic subgroups and some examples. In particular, we investigate the special subgroup $G_\Delta$ of $G_\Gamma$, where $V(\Delta)=\left\{v\in V(\Gamma)\mid st(v)=V(\Gamma)\right\}$ and give a condition on the vertex groups so that $G_\Delta$ is characteristic. We also consider a canonical map induced by the fact that $G_\Delta$ is characteristic and compute the kernel of this map which leads to the proof of Proposition E(ii).

We start Chapter 4 with a brief exposition of different types of automorphisms of a graph product whose vertex groups are cyclic. This chapter also contains a proof for the fact that torsion subgroups in $\Aut(G_\Gamma)$ are finite and $\Aut(G_\Gamma)$ does not contain a subgroup isomorphic to $\mathbb{Q}$. These two properties are important for the application to automatic continuity.

Chapter 5 contains the proof of Theorem B and the proofs of Corollaries C and D. Chapter 6 is dedicated  to the proof of Theorem G.

We start Chapter 7 with some known facts about the center of the automorphism group and move on to the calculation of the center of the automorphism group of a graph product of cyclic groups. 

The last chapter, Chapter 8, provides applications of Theorem G to automatic continuity and to reduced group $C^\ast$-algebras of automorphism groups of graph products.

\subsection*{Acknowledgments:}
We would like to thank Samuel M. Corson, Dominic Enders and Linus Kramer for
interesting discussions and remarks concerning automatic continuity. Additionally we would like to thank Daniel Keppeler, Luis Paris and Petra Schwer for many useful comments and remarks on an earlier version of this article. On top of that we want to thank Maria Gerasimova, Tim de Laat and Hannes Thiel for very helpful discussions regarding $C^\ast$-algebras.
\section{Graph products of groups}
In this section we first recall the basics of graph products of groups. This includes their definition, parabolic subgroups and results about subgroups of directly indecomposable graph products. We define what the term \textit{full-sized} means and then study normal subgroups of graph products and of their automorphism groups with regard to being finite or full-sized. To finish off the section we prove Proposition A.

In group theory it is common to build new groups out of given ones using free or direct product constructions. Further, we can also use amalgamation to obtain a new group out of given ones. For example,  $\Z/4\Z*\Z/6\Z$ is an infinite group, $\Z/4\Z\times \Z/6\Z$ is a finite abelian group and $\Z/4\Z*_{\Z/2\Z} \Z/6\Z\cong{\rm SL}_2(\Z)$. Graph products generalize  free/direct products and special kind of amalgamations. We start by recalling the definition of a graph product of groups.

\begin{definition}
Given a finite simplicial graph $\Gamma$ and a collection of groups $\mathcal{G} = \{ G_u \mid u \in V(\Gamma) \}$ indexed by the vertex-set $V(\Gamma)$ of $\Gamma$, the \emph{graph product} $G_\Gamma$ is defined as the quotient
$$ \left( \underset{u \in V(\Gamma)}{\ast} G_u \right) / \langle \langle [g,h]=1, \ g \in G_u, h \in G_v, \{u,v\} \in E(\Gamma) \rangle \rangle$$
where $E(\Gamma)$ denotes the edge-set of $\Gamma$.
\end{definition}

$ $\\{\bf Convention:} All vertex groups are always non-trivial. If a vertex group $G_v$ is cyclic for a vertex $v$, we also write $v$ for a generator of $G_v$ by slight abuse of notation.

\vspace{0.5cm}
Graph products of groups were introduced and studied by Green in her PhD Thesis \cite{Green}. If the vertex groups are all cyclic of the same type, then it is common to use another name for the graph product.
\begin{definition}
	Let $G_\Gamma$ be a graph product.
\begin{enumerate}
	\item If all vertex groups are infinite cyclic, then $G_\Gamma$ is called a \emph{right-angled Artin group} which we denote by $A_\Gamma$.
	\item If all vertex groups are cyclic of order $2$, then $G_\Gamma$ is called a \emph{right-angled Coxeter group} which we denote by $W_\Gamma$.
\end{enumerate}
\end{definition}

Given a simplicial graph $\Gamma$ and a vertex $v\in V(\Gamma)$ we define the \emph{link of $v$} as follows $lk(v):=\left\{w\in V(\Gamma)\mid \left\{v, w\right\}\in E(\Gamma)\right\}$ and the \emph{star of $v$} is defined as $st(v):=\left\{v\right\}\cup lk(v)$. The valency of $v$ is defined as $val(v):=|\left\{w\in V(\Gamma)\mid \left\{v,w\right\}\in E(\Gamma)\right\}|$. A subgraph 
$\Omega\subseteq\Gamma$ is called \emph{induced} if for all pair of vertices $(v,w)\in V(\Omega)\times V(\Omega)$ we have $\left\{v,w\right\}\in E(\Omega)$ if and only if $\left\{v,w\right\}\in E(\Gamma)$. 

Given a graph product $G_\Gamma$ and an induced subgraph $\Omega\subseteq\Gamma$, then the subgroup $\langle G_v\mid v\in V(\Omega)\rangle\subseteq G_\Gamma$ is canonically isomorphic to $G_\Omega$ (see \cite[Lemma 3.20]{Green}) and is called a \emph{special parabolic subgroup}.  Further, a conjugate of a special parabolic subgroup is called a \emph{parabolic subgroup}. 

By definition, a graph product $G_\Gamma$ is called \emph{directly decomposable} if $G_\Gamma$ is a direct product of proper special parabolic subgroups of $G_\Gamma$. For example, if $\Gamma$ is a cycle of length $n\geq 3$, then $G_\Gamma$ is directly decomposable if and only if $n=3$ or $n=4$. 

Further, a graph $\Gamma$ is a \emph{join} if there exists a partition $V(\Gamma)=A\cup B$, where $A$
and $B$ are both non-empty, such that for every  $(a,b)\in A\times B$  we have $\left\{a,b\right\}\in E(\Gamma)$. If $\Gamma$ is a join, then we write $\Gamma=\Gamma_1*\Gamma_2$ where $\Gamma_1$ is the induced subgraph by $A$ and $\Gamma_2$ is the induced subgraph by $B$. Hence, $G_\Gamma$ is directly decomposable if and only if $\Gamma$ is a join and then $G_\Gamma\cong G_{\Gamma_1}\times G_{\Gamma_2}$.
\vspace{0.5cm}

Since the main focus of this article is on normal subgroups, we  ask the following question:
Let $G_\Gamma$ be a graph product of arbitrary groups and $G_\Omega$ be a special parabolic subgroup. Under which combinatorial conditions on the graph $\Gamma$ is $G_\Omega$ normal in $G_\Gamma$?
The answer to this question was given by Antolin and Minasyan in \cite{AntolinMinasyan}. They showed, that $G_\Omega$ is normal in $G_\Gamma$ if and only if $G_\Omega$ is a direct factor in a direct decomposition of $G_\Gamma$ in special parabolic subgroups (see \cite[Prop. 3.13]{AntolinMinasyan}). In particular, $G_\Omega\trianglelefteq G_\Gamma$ if and only if $V(\Omega)\subseteq lk(v)$ for all $v\in V(\Gamma)-V(\Omega)$.

Given a non-trivial subgroup $H\subseteq G_\Gamma$ one can consider parabolic subgroups $gG_\Omega g^{-1}$ that contain $H$ as a subgroup. The natural question regarding these parabolic subgroups is if their intersection is again parabolic.
\begin{proposition}(\cite[Prop. 3.10]{AntolinMinasyan})
    Let $G_\Gamma$ be a graph product. For every non-trivial subgroup $H\subseteq G_\Gamma$ there exists a unique minimal (with respect to inclusion) parabolic subgroup $gG_\Omega g^{-1}$ such that $H\subseteq gG_\Omega g^{-1}$. This parabolic subgroup is called the parabolic closure of $H$ and is denoted by $\PC(H)$.
\end{proposition}
Let us consider the right-angled Coxeter group $W_\Gamma$ defined by the graph $\Gamma$ in Figure 2.

\begin{figure}[h]
	\begin{center}
		\begin{tikzpicture}
			\draw[fill=black]  (0,0) circle (3pt);
			\node at (0,-0.4) {$v_1$};
			\draw[fill=black]  (2,0) circle (3pt);
			\node at (2,-0.4) {$v_2$};
			\draw[fill=black]  (4,0) circle (3pt);
			\node at (4,-0.4) {$v_3$};
			\draw (0,0)--(2,0);
			\draw (2,0)--(4,0);
\end{tikzpicture}
\caption{Graph $\Gamma$.}
\end{center}
\end{figure}
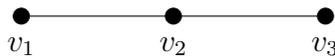
We have $W_\Gamma\cong(\langle v_1\rangle*\langle v_3\rangle)\times\langle v_2\rangle\cong (\Z/2\Z*\Z/2\Z)\times\Z/2\Z$. Since the infinite cyclic subgroup $\langle v_1v_3\rangle$ is normal in $\langle v_1\rangle *\langle v_3\rangle$, we know that $\langle v_1v_3\rangle$ is also normal in the entire Coxeter group $W_\Gamma$. We note also that $\PC(\langle v_1v_3\rangle)=\langle v_1\rangle*\langle v_3\rangle$ and furthermore $\PC(\langle v_1v_3\rangle)$ is also normal in $W_\Gamma$. In general, we have

\begin{lemma}
	\label{PCNormal}
	Let $G_\Gamma$ be a graph product and $N\subseteq G_\Gamma$ be a subgroup. If $N\trianglelefteq G_\Gamma$ is normal, then the parabolic closure $\PC(N)\trianglelefteq G_\Gamma$ is normal and $\PC(N)$ is a direct factor in a direct decomposition of $G_\Gamma$ in special parabolic subgroups.
	
	In particular, if $G_\Gamma$ is directly indecomposable, then $\PC(N)=G_\Gamma$.
\end{lemma}
\begin{proof}
	First, we remark that the normalizer of a subgroup $H$ in $G_\Gamma$ is contained in the normalizer of the parabolic closure of $H$, see \cite[Lemma 3.12]{AntolinMinasyan}. 
	
	Thus, given a normal subgroup $N\trianglelefteq G_\Gamma$ we know that $\PC(N)\trianglelefteq G_\Gamma$ is normal. Further, Proposition 3.13 in \cite{AntolinMinasyan} tells us that $\PC(N)$ is a direct factor of $G_\Gamma$. More precisely, $G_\Gamma=\PC(N)\times G_\Omega$, where $G_\Omega$ is a special parabolic subgroup of $G_\Gamma$.
\end{proof}

In the special case where $G_\Gamma$ is directly indecomposable we have a description of all subgroups of $G_\Gamma$. In particular of those subgroups without $F_2$ as a subgroup. 
\begin{proposition}\cite[Thm. 4.1]{AntolinMinasyan}
	\label{TitsAlternative}
	Let $G_\Gamma$ be a graph product of arbitrary groups and $H\subseteq G_\Gamma$ be a subgroup. If $G_\Gamma$ is directly indecomposable and $|V(\Gamma)|\geq 2$, then 
	\begin{enumerate}
		\item $H$ is contained in a proper parabolic subgroup or
		\item $H\cong\Z$ or
		\item $H\cong \Z/2\Z*\Z/2\Z$ or
		\item $H$ has a subgroup isomorphic to $F_2$.
	\end{enumerate}
\end{proposition}

The following property of a group is important for us.
\begin{definition}
    A group $G$ is called \emph{full-sized} if it contains a non-abelian free group $F_2\hookrightarrow G$.
\end{definition}

Now, the natural question for us is
\begin{question}
	Let $G_\Gamma$ be a (directly indecomposable) graph product of arbitrary groups and $N\trianglelefteq G_\Gamma$ be a non-trivial normal subgroup. Assume that $N$ is non full-sized. What can we say about the structure of $\Gamma$?  
\end{question}

To answer the question we need to recall elementary Bass-Serre theory. More precisely we need the fact that associated to every amalgamated product $A\ast_C B$ there is an associated tree, $T_{A\ast_C B}$ on which $A\ast_C B$ acts without a global fixed point. The vertices of the tree correspond to the cosets $gA$ and $hB$ and there is an edge between $gA$ and $hB$ whenever their intersection is equal to $gC$. Clearly, the amalgamated product $A*_C B$ acts via left multiplication on $T_{A\ast_C B}$. Most important for us is that the stabilizers of vertices are conjugates of $A$ and $B$ respectively. A good reference concerning Bass-Serre theory is \cite{Serre}. 

Before we move on to the next proposition, let us discuss an example of a Bass-Serre tree. The Coxeter group $W_\Gamma$ associated to the graph $\Gamma$ in Figure 2 can also be written as an amalgamated product as follows 
$$W_\Gamma\cong \langle v_1, v_2\rangle*_{\langle v_2\rangle}\langle v_2, v_3\rangle.$$
We define $A:=\langle v_1, v_2\rangle, B:=\langle v_2, v_3\rangle$ and $C:=\langle v_2\rangle$. Figure 3 shows a part of the infinite tree $T_{A\ast_C B}$. Note that each vertex in $T_{A\ast_C B}$ has valency two. Let us consider the action of $A\ast_C B$ via left multiplication on  $T_{A*_C B}$. For example, $v_2$ acts trivially on $T_{A\ast_C B}$, $v_1$ acts as an elliptic isometry with the fixed point set consisting of the single vertex labeled by $A$ and $v_1v_3$ acts as a translation. Moreover, the isometry group of $T_{A\ast_C B}$ is isomorphic to the infinite dihedral group.
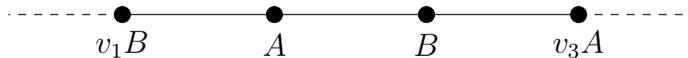
\begin{figure}[h]
	\begin{center}
		\begin{tikzpicture}
			\draw[fill=black]  (0,0) circle (3pt);
			\node at (0,-0.4) {$v_1B$};
			\draw[fill=black]  (2,0) circle (3pt);
			\node at (2,-0.4) {$A$};
			\draw[fill=black]  (4,0) circle (3pt);
			\node at (4,-0.4) {$B$};
			\draw[fill=black]  (6,0) circle (3pt);
			\node at (6,-0.4) {$v_3A$};
			\draw (0,0)--(2,0);
			\draw (2,0)--(4,0);
			\draw (4,0)--(6,0);
			\draw[dashed] (0,0)--(-1.5,0);
			\draw[dashed] (6,0)--(7.5,0);
\end{tikzpicture}
\caption{A part of the Bass-Serre tree $T_{A*_C B}$.}
\end{center}
\end{figure}

We note also that the Bass-Serre tree associated to the free product $\Z/n\Z*\Z/n\Z$ is infinite and each vertex has valency $n$.

\begin{proposition}
	\label{NormalGraphProduct}
	Let $G_\Gamma$ be a directly indecomposable graph product with $|V(\Gamma)|\geq 2$ and $N\trianglelefteq G_\Gamma$ be a non-trivial normal subgroup. If $N$ is non full-sized, then $G_\Gamma\cong\Z/2\Z*\Z/2\Z$.
\end{proposition}
\begin{proof}
   First we note that if a group $G$ has a normal subgroup $N\cong \Z/2\Z*\Z/2\Z=\langle a,b\mid a^2=b^2=1\rangle$, then the subgroup $\langle ab\rangle\cong\Z$ is normal in $G$. This is the case since the only elements of infinite order in $\Z/2\Z*\Z/2\Z$ are of the form $(ab)^k$ for some $k\in\mathbb{Z}-\{0\}$ and conjugation does not change the order of an element.
   
   Let $N$ be a non-trivial non full-sized normal subgroup in $G_\Gamma$. By Proposition \ref{TitsAlternative} we know that $N\cong\mathbb{Z}$ or $N\cong\Z/2\Z*\Z/2\Z$. If $N\cong \Z/2\Z*\Z/2\Z$ then by the above remark we know that $G_\Gamma$ has a normal subgroup $N'\cong\Z$. Our goal is to show that $G_\Gamma$ is isomorphic to $\Z/2\Z*\Z/2\Z$.
   
   By assumption $\Gamma$ is not complete and has at least $2$ vertices, therefore $G_\Gamma$ decomposes as an amalgamated product of special parabolic subgroups $G_\Gamma\cong G_{\Gamma_1}*_{G_{\Gamma_3}}G_{\Gamma_2}$. More precisely, we choose a vertex $v\in V(\Gamma)$ and then $V(\Gamma_1)=V(\Gamma)-\left\{v\right\}$, $V(\Gamma_3)=lk(v)$ and $V(\Gamma_2)=st(v)$. By Bass-Serre theory \cite{Serre}, there is a simplicial tree $T$ on which $G_{\Gamma_1}*_{G_{\Gamma_3}}G_{\Gamma_2}$ acts simplicially without a global fixed point. We restrict this action to the infinite cyclic normal subgroup $N$ (resp. $N'$). The generator $n$ of this normal subgroup must act via a hyperbolic isometry since otherwise $N$ has a global fixed point and hence is contained in a proper parabolic subgroup. In that case we would obtain $\PC(N)\neq G_\Gamma$, which is impossible by Lemma \ref{PCNormal}. Hence, we have a group homomorphism $N\to{\rm Isom}({\rm Min}(N))\cong\Z/2\Z*\Z/2\Z$, where ${\rm Min}(N)$ denotes the straight path in $T$ that is invariant under $N$ and on which $n$ induces a translation of non-zero amplitude (see \cite[I 6.4 Prop. 25]{Serre}). Since $N$ is normal and for $g\in G_\Gamma$ we have $g{\rm Min}(N)={\rm Min}(gNg^{-1})={\rm Min}(N)$, we obtain a group homomorphism $\varphi\colon G_\Gamma\to{\rm Isom(Min}(N))\cong \Z/2\Z*\Z/2\Z$. Furthermore, we know that   $\ker(\varphi)$ has a global fixed point and is therefore contained in a proper parabolic subgroup. By the previous arguments it follows that $\ker(\varphi)=\left\{1\right\}$. More precisely, since $\ker(\varphi)$ is normal, $\PC(\ker(\varphi))$ is also normal. We know that $G_\Gamma$ is directly indecomposable, so by Lemma \ref{PCNormal} $\PC(\ker(\varphi))$ is either $G_\Gamma$ or trivial. We additionally have $\ker(\varphi))\subseteq \PC(\ker(\varphi))=\{1\}$, because it is also contained in a proper parabolic subgroup. Hence $\varphi$ is injective. The map $\varphi$ is also surjective since $\Gamma$ has at least $2$ vertices. Thus $G_\Gamma\cong \Z/2\Z*\Z/2\Z$.
\end{proof}

Using Proposition \ref{NormalGraphProduct} we are able to characterize non full-sized normal subgroups in  directly decomposable graph products. But first, we need to recall the connection between a combinatorial property of the defining graph $\Gamma$ and the center of $G_\Gamma$. 

The center of a graph product can be easily read of the defining graph $\Gamma$ if the vertex groups are all abelian. More precisely,
a graph product $G_\Gamma$ decomposes as a direct product of special parabolic subgroups $G_\Gamma\cong G_\Delta\times G_{\Gamma-\Delta}$ where $V(\Delta)=\left\{v\in V(\Gamma)\mid st(v)=V(\Gamma)\right\}$. Furthermore, the center of $G_\Gamma$ is contained in the special parabolic subgroup $G_\Delta$ and if the vertex groups in $V(\Delta)$ are abelian, then $Z(G_\Gamma)\cong G_\Delta$.

\begin{corollary}
\label{ProductNormal}
 	Let $G_\Gamma$ be a graph product. We decompose $G_\Gamma\cong G_{\Gamma_1}\times\ldots\times G_{\Gamma_n}$ as a direct product where the factors are directly indecomposable special parabolic subgroups. Further, we define $$J:=\left\{j\in\left\{1,\ldots, n\right\}\mid |V(\Gamma_j)|\geq 2, G_{\Gamma_j}\cong \Z/2\Z*\Z/2\Z\right\}.$$ Let $N\trianglelefteq G_\Gamma$ be a non-trivial normal subgroup. If $N$ is non full-sized, then $N\trianglelefteq G_\Delta\times \prod_{j\in J} G_{\Gamma_j}$ where $V(\Delta)=\left\{v\in V(\Gamma)\mid st(v)=V(\Gamma)\right\}$ and the image of the projection of $N$ to $G_{\Gamma_j}$ is trivial or infinite for all $j\in J$. In particular, if $N$ is finite, then $N\trianglelefteq  G_\Delta$.
\end{corollary}
\begin{proof}
	First, we note that if a group $G$ is non full-sized, then every quotient of $G$ is non full-sized.
	
	Let $N\trianglelefteq G_\Gamma$ be a non full-sized normal subgroup. Now, let us consider the projection $\pi_i\colon G_\Gamma\twoheadrightarrow G_{\Gamma_i}$. Since $\pi_i$ is surjective, we know that $\pi_i(N)\trianglelefteq G_{\Gamma_i}$. Furthermore, by the above remark  the normal subgroup $\pi_i(N)$ is non full-sized. Therefore Proposition \ref{NormalGraphProduct} implies that if $|V(\Gamma_i)|\geq 2$ and $G_{\Gamma_i}$ is not isomorphic to $\Z/2\Z*\Z/2\Z$, then $\pi_i(N)$ is trivial. Thus $N\trianglelefteq G_\Delta\times \prod_{j\in J} G_{\Gamma_j}$ where $V(\Delta)=\left\{v\in V(\Gamma)\mid st(v)=V(\Gamma)\right\}$ and $$J:=\left\{j\in\left\{1,\ldots, n\right\}\mid |V(\Gamma_j)|\geq 2, G_{\Gamma_j}\cong \Z/2\Z*\Z/2\Z\right\}.$$ 
	Since $\pi_j(N)\trianglelefteq G_{\Gamma_j}\cong \Z/2\Z*\Z/2\Z$, we know that $\pi_j(N)\cong\mathbb{Z}$ or $\pi_j(N)\cong\Z/2\Z*\Z/2\Z$.  Hence, if $N$ is finite, then $N\trianglelefteq G_\Delta$.
\end{proof}

Since the center of a graph product of abelian groups is isomorphic to $G_\Delta$, the following result is a direct consequence of Corollary \ref{ProductNormal}.
\begin{corollary}
\label{NotDInfty}
Let $G_\Gamma$ be a graph product of abelian groups. We decompose $G_\Gamma\cong G_{\Gamma_1}\times\ldots\times G_{\Gamma_n}$ as a direct product where the factors are directly indecomposable special parabolic subgroups. 

If $G_{\Gamma_i}$ is not isomorphic to $\Z/2\Z*\Z/2\Z$ for all $i=1,\ldots, n$, then a normal subgroup in $G_\Gamma$ is contained in the center of $G_\Gamma$ or is full-sized.
\end{corollary}

As an immediate corollary we obtain the following result concerning normal subgroups in right-angled Artin groups. This result can also be proven directly as follows using \cite[Thm. 1.2]{Baudisch}.
\begin{lemma}(\cite{math})
	\label{F2}
	Let $A_\Gamma$ be a right-angled Artin group and $N$ be a subgroup. If $N$ is normal, then $N$ is contained in the center of $A_\Gamma$ or $N$ is full-sized.
\end{lemma}
\begin{proof}
	It was proven by Baudisch in \cite[Thm. 1.2]{Baudisch} that every two generated subgroup $\langle u, v\rangle$ in $A_\Gamma$ such that $u$ and $v$ do not commute is isomorphic to $F_2$.
	
	We decompose $A_\Gamma$ as $A_\Gamma\cong A_\Delta \times A_{\Gamma-\Delta}$. Note that the special parabolic subgroup $A_{\Gamma-\Delta}$ is centerless.
	
	If $N$ is not contained in the center of $A_\Gamma$, which is equal to $A_\Delta$, then
	 the normal subgroup  $N\cap A_{\Gamma-\Delta}$  of the special parabolic subgroup $A_{\Gamma-\Delta}$ is non-trivial. Thus there exist $n\in N\cap A_{\Gamma-\Delta}$, $n\neq 1$ and $x\in A_{\Gamma-\Delta}$ such that $xn\neq nx$. By Baudisch's result we know that $\langle x,n\rangle\cong F_2$. Therefore, the elements $n$ and  $xnx^{-1}$ do not commute either, and again by Baudisch's result we have $\langle n, xnx^{-1}\rangle\cong F_2$. Since  $N\cap A_{\Gamma-\Delta}$ is normal, if follows that $xnx^{-1}\in N\cap A_{\Gamma-\Delta}$. Thus $F_2\cong \langle n, xnx^{-1}\rangle\subseteq N\cap A_{\Gamma-\Delta}\subseteq N$.
\end{proof}
We note that the structure of finitely generated normal subgroups in right-angled Artin groups was investigated in \cite{CasalsRuizZearra}.

\subsection{Normal subgroups in automorphism groups}

For a group $G$ we denote by $\Aut(G)$ the group of all bijective group homomorphisms from $G$ to $G$ and by $\Inn(G)$ the normal subgroup consisting of all inner automorphisms. More precisely, for $g\in G$ we denote the inner automorphism that maps $h\in G$ to $ghg^{-1}$ by $\gamma_g$. Hence the map $\varphi\colon G\to \Inn(G)$, $\varphi(g)=\gamma_g$ is an epimorphism with $\ker(\varphi)=Z(G)$.  Thus, if $G$ is centerless and $\gamma_{g_1}=\gamma_{g_2}$, then the equality $g_1=g_2$ holds and therefore $\Inn(G)\cong G$.

We begin by proving the following result for the automorphism group of a centerless group.
\begin{lemma}
\label{GeneralCenterless}
    Let $\mathcal{P}$ be a property of groups that is inherited by subgroups. 
    Let $G$ be a centerless group. If $G$ does not have non-trivial normal subgroups with property $\mathcal{P}$, then $\Aut(G)$ also does not have non-trivial normal subgroups with property $\mathcal{P}$.
\end{lemma}
\begin{proof}
 Let $G$ be a centerless group that does not have non-trivial normal subgroups with property $\mathcal{P}$. Further, let $N$ be a normal subgroup in $\Aut(G)$ with property $\mathcal{P}$.
We consider the subgroup $N\cap\Inn(G)$ which is a normal subgroup in $\Inn(G)\cong G$ with property $\mathcal{P}$. Thus, by assumption $N\cap \Inn(G)$ is trivial.  Let $g\in G$ be an arbitrary element and $n\in N$ be an arbitrary element. 
	We have
	$$n\circ \gamma_g\circ n^{-1}=\gamma_{n(g)}.$$
	We multiply the above equation by $\gamma_{g^{-1}}$ on the right and obtain
	$$n\circ \gamma_g\circ n^{-1}\circ \gamma_{g^{-1}}=\gamma_{n(g)}\circ\gamma_{g^{-1}}.$$
	Since $N$ is normal, the left side of the equation in an element in $N$ and the right side of the equation is an element in $\Inn(G)$. Hence 
	$$\gamma_{n(g)}\circ\gamma_{g^{-1}}=id.$$
	Thus $n(g)g^{-1}\in Z(G)=\{1\}$ and $N$ is trivial.
\end{proof}

\begin{remark}
The proof of the above lemma shows the following: Given a centerless group $G$ and a normal subgroup $N\trianglelefteq\Aut(G)$, for every $n\in N$ and $g\in G$ the inner automorphism $\gamma_{n(g)g^{-1}}$ is contained in $N$.
\end{remark}

\begin{proposition}
\label{CenterlessAutG}
Let $G_\Gamma$ be a graph product of arbitrary groups. We decompose $G_\Gamma\cong G_{\Gamma_1}\times\ldots\times G_{\Gamma_n}$ as a direct product where the factors are directly indecomposable special parabolic subgroups. 
\begin{enumerate}
\item If $V(\Delta)=\emptyset$, then $\Aut(G_\Gamma)$ does not have non-trivial finite normal subgroups.
\item Assume that $G_{\Gamma_i}\ncong\Z/2\Z*\Z/2\Z$ for all $i=1,\ldots, n$.
If $|V(\Gamma_i)|\geq 2$ for $i=1\ldots,n$, then a non-trivial normal subgroup in $\Aut(G_\Gamma)$ is full-sized. 
In particular, if $G_\Gamma$ is a centerless graph product of abelian groups, then a non-trivial normal subgroup in $\Aut(G_\Gamma)$ is full-sized. 
\end{enumerate}
\end{proposition}
\begin{proof}
If $V(\Delta)=\emptyset$, then $Z(G_\Gamma)=\left\{1\right\}$. Corollary \ref{ProductNormal} says that $G_\Gamma$ does not have non-trivial finite normal subgroups. Since the property to be finite is inherited by subgroups, Lemma \ref{GeneralCenterless} implies that $\Aut(G_\Gamma)$ does not have non-trivial finite normal subgroups. 
	
Let $\mathcal{P}$ be the property of a group to be non full-sized. Clearly, $\mathcal{P}$ is inherited by subgroups. Now, Corollary \ref{ProductNormal} tells us that in case (ii) $G_\Gamma$ does not have non-trivial non full-sized normal subgroups. Thus, by Lemma \ref{GeneralCenterless} a non-trivial normal subgroup in $\Aut(G_\Gamma)$ is full-sized.	
\end{proof}

As an immediate corollary we obtain Proposition A from the introduction.
\vspace{0.5cm}

We end this chapter by recalling a description of non full-sized normal subgroups in the automorphism group of a graph product $G_\Gamma$ where $\Gamma$ is a complete graph and all vertex groups are infinite cyclic, in that case $\Aut(G_\Gamma)\cong\Aut(\Z^n)\cong \GL_n(\Z)$. Note that a subgroup in $\GL_n(\Z)$ is virtually solvable or is full-sized (see \cite{Tits}).

Let us first describe normal subgroups in $\GL_n(\Z)$ in the cases where $n=1, 2$. If $n=1$, then $\GL_1(\Z)\cong \Z/2\Z$. If $n=2$, then ${\rm PSL}_2(\Z)\cong \Z/2\Z*\Z/3\Z$
and Proposition \ref{NormalGraphProduct} then implies that a non-trivial normal subgroup in ${\rm PSL}_2(\Z)$ is full-sized. Hence a non-trivial non full-sized normal subgroup in $\GL_2(\Z)$ is equal to the center $Z(\GL_2(\Z))=\left\{ I, -I\right\}$ where $I$ is the identity matrix.  

Recall, a group $G$ is called \emph{almost simple} if every normal subgroup is either finite and is contained in the center or has finite index in $G$. It follows from the Margulis normal subgroups theorem \cite{Margulis} that ${\rm SL}_n(\Z)$ is almost simple for $n\geq 3$. Since ${\rm SL}_n(\Z)$ has index $2$ in $\GL_n(\Z)$ and every finite index subgroup in ${\rm SL}_n(\Z)$ is full-sized (see \cite[Exercise 4.E.20(4)]{Loeh}), we know that a non full-sized normal subgroup in $\GL_n(\Z)$ is contained in the center of $\GL_n(\Z)$. We summarize these results in the following lemma.
\begin{lemma}
\label{GLnoF2}
	Let $N$ be a non-trivial normal subgroup in ${\rm GL}_n(\mathbb{Z})$. If $N$ is non full-sized, then $N=Z({\rm GL}_n(\mathbb{Z}))=\left\{I, -I\right\}$.
\end{lemma}
\section{Characteristic subgroups}
To construct quotients of the entire automorphism group $\Aut(G_\Gamma)$ it is a common tool to consider characteristic subgroups of $G_\Gamma$. This can be helpful since, in some cases, it allows us to reduce a question about a complicated graph product to a question about two smaller graph products and a short exact sequence.

We start this chapter by recalling the definition of a characteristic subgroup. Using this we obtain a very important short exact sequence that will be used throughout the rest of the article. Furthermore we prove some basic but useful results about a splitting of the automorphism group of a graph product. Finally we show that if a right-angled Coxeter groups decomposes as a direct product of an infinite dihedral group and another right-angled Coxeter group, it has a characteristic subgroup isomorphic to a product of copies of the infinite dihedral group.

\begin{definition}
	Let $G$ be a group and $H\subseteq G$ be a subgroup. The group $H$ is called \emph{characteristic in $G$} if for every $f\in\Aut(G)$ the equality 
	$f(H)=H$ holds.
\end{definition}
A characteristic subgroup $H$ of a group $G$ is always normal in $G$ and we have two group homomorphisms: $\pi_1\colon\Aut(G)\to\Aut(H)$ and $\pi_2\colon\Aut(G)\to\Aut(G/H)$ defined as follows $\pi_1(f)=f_{|H}$ and
$\pi_2(f)(gH):=f(g)H$.

For a graph product $G_\Gamma$ we always have the direct decomposition $G_\Delta\times G_{\Gamma-\Delta}$ where $V(\Delta)=\left\{v\in V(\Gamma)\mid st(v)=V(\Gamma)\right\}$. If the vertex groups in $\Delta$ are abelian, then $Z(G_\Gamma)=G_\Delta$ is a characteristic subgroup of $G_\Gamma$ and we obtain two group homomorphisms
$$\pi_1\colon\Aut(G_\Gamma)\to\Aut(G_\Delta)\text{ and }\pi_2\colon \Aut(G_\Gamma)\to\Aut(G_{\Gamma-\Delta})$$
which are obviously surjective. But if there exists a vertex group in $\Delta$ that is not abelian, then $G_\Delta$ is in general not characteristic. For example, consider the graph product of the defining graph $\Gamma$ in Figure 4.

\begin{figure}[h]
	\begin{center}
		\begin{tikzpicture}
			\draw[fill=black]  (0,0) circle (3pt);
			\node at (0,-0.4) {$\Z/2\Z$};
			\draw[fill=black]  (2,0) circle (3pt);
			\node at (2,-0.4) {$\Z/2\Z*\Z/2\Z$};
			\draw[fill=black]  (4,0) circle (3pt);
			\node at (4,-0.4) {$\Z/2\Z$};
			\draw (0,0)--(2,0);
			\draw (2,0)--(4,0);
\end{tikzpicture}
\caption{ $G_\Gamma\cong(\Z/2\Z*\Z/2\Z)^2$.}
\end{center}
\end{figure}
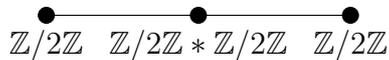
It is obvious, that the subgroup $G_\Delta\cong \Z/2\Z*\Z/2\Z$ is not characteristic. This observation leads us to the following question.
\begin{question}
Let $G_\Gamma$ be a graph product. Under which conditions on the vertex groups in $\Delta$ is the subgroup $G_\Delta$ characteristic in $G_\Gamma$?
\end{question}

Following Genevois \cite{Genevois}, a group $G$ is called \emph{graphically irreducible} if, for every finite simplicial graph $\Gamma$ and
every collection of vertex groups indexed by $V(\Gamma)$ such that $G\cong G_\Gamma$, the graph $\Gamma$ is a clique.

\begin{lemma}(\cite[Prop. 3.11]{Genevois})
	Let $G_\Gamma$ be a graph product. If the vertex groups are graphically irreducible, then the special parabolic subgroup $G_\Delta$ is characteristic in $G_\Gamma$.
\end{lemma}

Hence, we obtain two surjective group homomorphisms 
$$\pi_1\colon \Aut(G_\Gamma)\to\Aut(G_\Delta)\text{ and }\pi_2\colon\Aut(G_\Gamma)\to \Aut(G_{\Gamma-\Delta}).$$ We define 
$$\pi\colon\Aut(G_\Gamma)\to\Aut(G_\Delta)\times\Aut(G_{\Gamma-\Delta})\text{ as follows }\pi(f)=(\pi_1(f), \pi_2(f)).$$ 

The group homomorphism $\pi$ is surjective by construction. Furthermore, the following proposition gives a concrete description of the kernel of $\pi$. 
\begin{proposition}
\label{graphicallyirreducible}
Let $G_\Gamma$ be a graph product of graphically irreducible vertex groups. The kernel of $\pi\colon\Aut(G_\Gamma)\to\Aut(G_\Delta)\times \Aut(G_{\Gamma-\Delta})$ is equal to 
$$T_\Delta:=\left\{f\in \Aut(G_\Gamma)\mid\text{ for all } g\in G_\Gamma: g^{-1}f(g)\in G_\Delta \text{ and }  f_{|G_\Delta}=id \right\}.$$ 
Moreover, the short exact sequence
$$\left\{1\right\}\to T_\Delta\to\Aut(G_\Gamma)\to \Aut(G_\Delta)\times \Aut(G_{\Gamma-\Delta})\to\left\{1\right\}$$ splits and we have
$$\Aut(G_\Gamma)\cong T_\Delta\rtimes (\Aut(G_\Gamma)\times \Aut(G_{\Gamma-\Delta})).$$	
\end{proposition}
\begin{proof}
	We follow the ideas of the proof of \cite[Lemma 7.7]{Leder}.
	
	Let $f\in\ker(\pi)=\ker(\pi_1)\cap\ker(\pi_2)$. Since $f\in\ker(\pi_1)$ we have $f(g)=g$ for all $g\in G_\Delta$. Further, since $f\in\ker(\pi_2)$ we have for all $h\in G_\Gamma$ 
	$$\pi_2(f)(hG_\Delta)=f(h)G_\Delta=hG_\Delta.$$ 
	Therefore $h^{-1}f(h)\in G_\Delta$ and hence $f\in T_\Delta$.
	
	It follows by definition of $T_\Delta$ that $T_\Delta\subseteq \ker(\pi)$.
	Hence we obtain a short exact sequence
	$$\left\{1\right\}\to T_\Delta\to\Aut(G_\Gamma)\to \Aut(G_\Delta)\times \Aut(G_{\Gamma-\Delta})\to\left\{1\right\}$$ 
	that splits, since for the canonical inclusion $i\colon \Aut(G_\Delta)\times \Aut(G_{\Gamma-\Delta})\to \Aut(G_\Gamma)$ we have $\pi\circ i=id$. Thus
	 $\Aut(G_\Gamma)\cong T_\Delta\rtimes (\Aut(G_\Gamma)\times \Aut(G_{\Gamma-\Delta})).$
\end{proof}

Now we collected all tools to prove Proposition E(ii).
\begin{proposition}
    \label{FiniteVertexGroups}
Let $G_\Gamma$ be a graph product. 
We decompose $G_\Gamma\cong G_{\Gamma_1}\times\ldots\times G_{\Gamma_n}$ as a direct product where the factors are directly indecomposable special parabolic subgroups. Assume that $G_{\Gamma_i}\ncong\Z/2\Z*\Z/2\Z$ for all $i=1,\ldots, n$.

If all vertex groups are finite, then a normal subgroup in $\Aut(G_\Gamma)$ is finite or full-sized.
\end{proposition}
\begin{proof}
    First we note that if $V(\Delta)$ is empty then by Proposition \ref{CenterlessAutG}(i) a non-trivial normal subgroup in $\Aut(G_\Gamma)$ is full-sized. 
    
    Let $N\trianglelefteq\Aut(G_\Gamma)$ be an infinite normal subgroup. By assumption all vertex groups are finite, thus in particular all vertex groups are graphically irreducible. By Proposition \ref{graphicallyirreducible} we have the exact sequence
    $$\left\{1\right\}\to T_\Delta\to\Aut(G_\Gamma)\overset{\pi=(\pi_1\times\pi_2)}{\rightarrow} \Aut(G_\Delta)\times \Aut(G_{\Gamma-\Delta})\to\left\{1\right\}.$$
    
    We claim that $N$ is full-sized. Indeed, if $\pi_2(N)\trianglelefteq \Aut(G_{\Gamma-\Delta})$ is non-empty, we know that $\pi_2(N)$ is full-sized and therefore $N$ is also full-sized.
    
    Now we turn to the case where $\pi_2(N)$ is trivial, then we have the following exact sequence
    $$\left\{1\right\}\to N\cap T_\Delta\to N\to \pi_1(N)\to\left\{1\right\}.$$
    The group $\Aut(G_\Delta)$ is finite because by assumption the vertex groups of $\Delta$ are finite. Thus $\pi_1(N)\subseteq\Aut(G_\Delta)$ is a finite group. Further, the group $N\cap T_\Delta$ is also finite, because $T_\Delta$ is a finite group. Hence $N$ is a finite normal subgroup.
\end{proof}

The next result will be useful in the proof of Theorem G. Namely,
in the case where the graph product $G_\Gamma$ decomposes as a direct product of characteristic subgroups, we completely understand the structure of $\Aut(G_\Gamma)$ via the automorphism groups of direct factors of $G_\Gamma$.
\begin{lemma}
\label{CharacteristicBothFactors}
	Let $G$ be a group that decomposes as direct product $G=G_1\times G_2$. If both factors $G_1$ and $G_2$ are characteristic subgroups of $G$, then ${\rm Aut}(G)\cong{\rm Aut}(G_1)\times{\rm Aut}(G_2)$.
\end{lemma}
\begin{proof}
    It is obvious that $\Aut(G_1)\times {\rm Aut}(G_2)$ can naturally be embedded in ${\rm Aut}(G)$ via the natural map $(f_1,f_2)\mapsto f$ where $f(v)=f_i(v)$ whenever $v\in G_i$ for $i=1,2$.
    
    We claim this map is surjective. Given a general $f\in{\rm Aut}(G)$, we know that $f(G_i)=G_i$ since both factors are characteristic. Hence we obtain $f_i\in {\rm Aut}(G_i)$ by setting $f_i:=f_{|G_i}$. It is easy to see that $(f_1,f_2)$ is mapped to $f$ under the natural embedding which is what we wanted to show. 
\end{proof}

Let us discuss one example of the above lemma in the context of graph products and their automorphism groups.
\begin{lemma}
\label{InfiniteFiniteCharacteristic}
    Let $G_\Gamma$ be a graph product of cyclic groups. If all vertex groups in $V(\Delta)$ are infinite cyclic and all vertex groups in $V(\Gamma)-V(\Delta)$ are finite, then $G_{\Gamma-\Delta}$ is characteristic and thus $\Aut(G_\Gamma)\cong\Aut(G_\Delta)\times \Aut(G_{\Gamma-\Delta})$.
\end{lemma}
\begin{proof}
    Let $f\in\Aut(G_\Gamma)$ and $w_i\in V(\Gamma)-V(\Delta)$. To show that $G_{\Gamma-\Delta}$ is characteristic, it is sufficient to prove that $f(w_i)\in G_{\Gamma-\Delta}$. It was proven in \cite[Thm. 3.26]{Green} that a finite order element is contained in a parabolic subgroup whose defining graph is a clique. Thus $f(w_i)\in gG_\Omega g^{-1}$ where $\Omega$ is a clique subgraph of $\Gamma-\Delta$ and $g\in G_\Gamma$. So there exists an element $h\in G_\Omega$ such that $f(w_i)=ghg^{-1}$. Since $G_\Delta$ commutes with $G_\Omega$ we can assume that $g\in G_{\Gamma-\Delta}$. Hence $ghg^{-1}\in G_{\Gamma-\Delta}$. 
    
    By Lemma \ref{CharacteristicBothFactors} it follows $\Aut(G_\Gamma)\cong\Aut(G_\Delta)\times \Aut(G_{\Gamma-\Delta})$.
\end{proof}

We end this chapter by an investigation of characteristic subgroups in right-angled Coxeter groups. Given a right-angled Coxeter group $W_\Gamma$, it was proven by Leder in \cite[Prop. 7.17]{Leder} that for $k\in\mathbb{N}$ the normal closure of the special subgroup $W_\Omega$ where $V(\Omega)=\left\{v\in V(\Gamma)\mid val(v)\geq k\right\}$ is a characteristic subgroup.

Let us consider the right-angled Coxeter group defined by the following graph $\Gamma$.

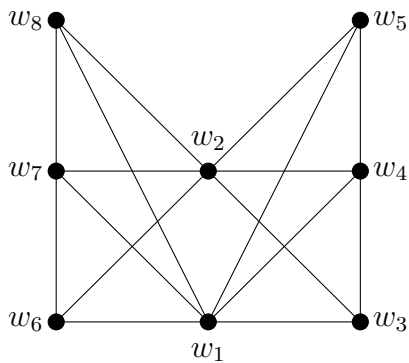
\begin{figure}[h]
	\begin{center}
		\begin{tikzpicture}
			\draw[fill=black]  (0,0) circle (3pt);
			\node at (0,-0.4) {$w_1$};
			\draw[fill=black]  (0,2) circle (3pt);
			\node at (0,2.4) {$w_2$};
			\draw[fill=black]  (2,0) circle (3pt);
			\node at (2.4,0) {$w_3$};
			\draw[fill=black]  (2,2) circle (3pt);
			\node at (2.4,2) {$w_4$};
			\draw[fill=black]  (2,4) circle (3pt);
		    \node at (2.4,4) {$w_5$};
			\draw (0,0)--(2,0);	
			\draw (0,0)--(2,2);	
			\draw (0,0)--(2,4);
			\draw (0,2)--(2,0);
			\draw (0,2)--(2,2);	
			\draw (0,2)--(2,4);	
			\draw (2,0)--(2,4);
			\draw[fill=black]  (-2,0) circle (3pt);
			\node at (-2.4,0) {$w_6$};
			\draw[fill=black]  (-2,2) circle (3pt);
			\node at (-2.4,2) {$w_7$};
			\draw[fill=black]  (-2,4) circle (3pt);
		    \node at (-2.4,4) {$w_8$};
		    \draw (-2,0)--(-2,4);
		    \draw (-2,0)--(0,0);
		    \draw (-2,0)--(0,2);
		    \draw (-2,2)--(0,0);
		    \draw (-2,2)--(0,2);
		    \draw (-2,4)--(0,0);
		    \draw (-2,4)--(0,2);
		\end{tikzpicture}
	\end{center}
\caption{Graph $\Gamma$.}
\end{figure}

The vertices $w_1$ and $w_2$ have valency $6$ and all the other vertices have smaller valencies. Thus the normal closure of the subgroup $\langle w_1, w_2\rangle$ is characteristic in $W_\Gamma$. Further, the special parabolic subgroup $\langle w_1, w_2\rangle\cong\Z/2\Z* \Z/2\Z$ is a direct factor of $W_\Gamma$, thus the normal closure of $\langle w_1, w_2\rangle$ is equal to $\langle w_1, w_2\rangle$ and this special parabolic subgroup is characteristic in $W_\Gamma$.

\begin{lemma}
Let $W_\Gamma$ be a centerless right-angled Coxeter group. We decompose $W_\Gamma$ as a direct product of directly indecomposable special parabolic subgroups $W_\Gamma=W_{\Gamma_1}\times\ldots\times W_{\Gamma_n}$. 

If at least one of the factors is isomorphic to $\Z/2\Z*\Z/2\Z$, then $\prod_{j\in J}W_{\Gamma_j}$ where $J:=\left\{j\in\left\{1,\ldots,n\right\}\mid W_j\cong \Z/2\Z*\Z/2\Z\right\}$ is a characteristic subgroup.

In particular, 	$\prod_{j\in J}W_{\Gamma_j}$ as a subgroup of ${\rm Inn}(W_\Gamma)$ is an infinite virtually abelian normal subgroup in $\Aut(W_\Gamma)$.
\end{lemma}
\begin{proof}
We write $V(\Gamma)=\left\{w_1,\ldots, w_n\right\}$.
First, we need to show the following combinatorial lemma from graph theory: Let $\Omega$ denote a finite simplicial graph with $n$ vertices. Assume that every vertex has at least valency $n-2$. Then $\Omega=A_1\ast A_2\ast...\ast A_k$ where every $A_j$ is an induced subgraph of $\Omega$ and has at most $2$ vertices for $j=1,...,k$.  

This is true since starting with a vertex $v_0\in V(\Omega)$, $v_0$ is connected to at least $n-2$ other vertices, of which there are $n-1$. If $v_0$ is connected to all of them, set $A_1=\{v_0\}$ and continue with $\Omega-\langle v_0\rangle$. If not, there is a vertex $v_1\in V(\Omega)$, which is not connected to $v_0$ by an edge. However $v_1$ has valency $n-2$, hence it is connected to all other vertices apart from $v_0$. In other words $\Omega=\langle v_0,v_1\rangle \ast (\Omega-\langle\{v_0,v_1\}\rangle)$. Since $\Omega$ is finite, this proves the combinatorial lemma in finitely many steps.

Note that this proof actually shows that the existence of two vertices $x,y$ which are not connected by an edge and which have valency $n-2$ always implies that $\Gamma$ decomposes as a join $\Gamma= (\{x,y\},\emptyset)\ast (\Gamma-(\{x,y\},\emptyset))$.

To show the result about right-angled Coxeter groups simply notice that given $W_{\Gamma_j}\cong \Z/2\Z\ast \Z/2\Z$ with canonical generators $w_i$ and $w_j$, the vertices $w_i$ and $w_j$ in the defining graph for the right-angled Coxeter group have valency $|V(\Gamma)|-2$. Furthermore there are no vertices with valency $|V(\Gamma)|-1$ since $W_\Gamma$ is centerless. Hence $\prod_{j\in J}W_{\Gamma_j}$ is nothing but $\langle V(D)\rangle $ where $D =\{v\in V(\Gamma)|val(v)\geq |V(\Gamma)|-2\}$. Thus,  \cite[Prop. 7.17]{Leder} implies that $\prod_{j\in J}W_{\Gamma_j}$ is a characteristic subgroup in $W_\Gamma$.

Furthermore, since $W_\Gamma$ is centerless we have $\Inn(W_\Gamma)\cong W_\Gamma$ and therefore we can consider the special parabolic subgroup $\prod_{j\in J}W_{\Gamma_j}$ as a subgroup of $\Inn(W_\Gamma)$. Let $f\in\Aut(W_\Gamma)$ and $\gamma_{w}\in\Inn(W_\Gamma)$ where $w\in\prod_{j\in J}W_{\Gamma_j}$.

We have 
$$f\circ \gamma_w\circ f^{-1}=\gamma_{f(w)}.$$
Since $\prod_{j\in J}W_{\Gamma_j}$ is a characteristic subgroup, we know that $f(w)\in \prod_{j\in J}W_{\Gamma_j}$. Hence  $\prod_{j\in J}W_{\Gamma_j}$ is a normal subgroup in $\Aut(W_\Gamma)$ and it is indeed virtually abelian since each direct factor is isomorphic to the infinite dihedral group $\Z/2\Z*\Z/2\Z$ which is virtually abelian.
\end{proof}
\section{The automorphism group of a graph product of cyclic groups}
In this chapter we specialise to the case of finitely generated abelian groups. First we remark that this is essentially the same as assuming the vertex groups are cyclic. Then we discuss certain types of automorphisms that typically occur in this situation and recall an important result by Corredor-Gutierrez which shows these automorphisms generate the entire automorphism group. Finally we show that torsion subgroups in the automorphism group of a graph product of cyclic groups are finite and that no subgroup is isomorphic to $\mathbb{Q}$.

Let $G_\Gamma$ be a graph product of finitely generated abelian groups. We remark that $G_\Gamma$ is isomorphic to
a graph product of cyclic groups. More precisely, given a finitely generated abelian vertex group $G_v$ we know that $G_v$ is isomorphic to the direct sum $\Z^l\oplus\Z/n_1\Z \oplus ... \oplus \Z/n_k\Z$. Thus replacing the vertex $v$ by a complete graph on $k+l$ vertices with vertex groups isomorphic to $\Z/n_i\Z$ or $\Z$ respectively yields an isomorphic graph product. Moreover this graph product is isomorphic to a graph product of cyclic groups where the finite groups have prime power order. In order to achieve that we simply replace the vertex with vertex group $\Z/n\Z$ where $n=p_1^{k_1}\cdot...\cdot p_r^{k_r}$ by a complete graph on $r$ vertices with vertex groups $\Z/p_j^{k_j}\Z$ where $p_i$ are prime numbers for $i=1,\ldots, r$.

It was proven by Corredor and Gutierrez in \cite{CorredorGutierrez} that the automorphism group of a graph product of cyclic groups is generated by a finite collection of different types of automorphisms: labeled graph automorphisms, local automorphisms, partial conjugations and dominated
transvections. In the following we explain these different types of automorphisms that will prove to be useful especially when calculating the center.

Let $G_\Gamma$ be a graph product of cyclic groups and $\Aut(G_\Gamma)$ the automorphism group of $G_\Gamma$. We start by introducing \emph{labeled graph automorphisms}. The group of labeled graph automorphisms is defined as follows $\Aut(\Gamma):=\left\{\sigma\in Sym(\Gamma)\mid G_v\cong G_{\sigma(v)}\text{ for all }v\in V(\Gamma)\right\}$. It is not hard to show that each element in $\Aut(\Gamma)$ canonically induces an automorphism in $\Aut(G_\Gamma)$, therefore we call also this induced automorphism a \emph{labeled graph automorphism}. 

Now we investigate automorphisms in $\Aut(G_\Gamma)$ that come from the automorphism group of the vertex groups. Given a non-identity automorphism in $\Aut(G_v)$ where $v\in V(\Gamma)$, this automorphism induces an automorphism of the entire graph product $G_\Gamma$ by sending the  generator $w$ of $G_w$, $v\neq w$ to itself. We call this kind of automorphism a \emph{local automorphism}. 

Let $v\in V(\Gamma)$ and let $C\subseteq\Gamma$ be a connected component of the graph induced by the vertex set $V(\Gamma)-st(v)$. We define a map $\rho_v^{C}\colon V(\Gamma)\to G_\Gamma$ as follows: $\rho_c^{C}(z)=vzv^{-1}$ for all $z\in V(C)$ and $\rho_x^C(w)=w$ for $w\in V(\Gamma)-V(C)$. This map induces an automorphism of $G_\Gamma$ which we call a \emph{partial conjugation}.

The last type of automorphisms that we are going to introduce are dominated transvections. Let $v, w \in V(\Gamma)$ be two different vertices. We define a map $\rho_{vw^j}\colon V(\Gamma)\to G_\Gamma$ as follows: $\rho_{vw^k}(v)=vw^k$ and identity on all other vertices in $V(\Gamma)$. We should be careful here, since this map does not always induce a well-defined automorphism of $G_\Gamma$. But 
if $ord(v)=\infty$ and $lk(v)\subseteq st(w)$, then $\rho_{vw}$ induces  a well-defined automorphism of $G_\Gamma$.
Further, if $ord(v)=p^i$ and $ord(w)=p^j$ and $st(v)\subseteq st(w)$, then we have two cases:
if $j\leq i$, then $\rho_{vw}$ is a well-defined automorphism of $G_\Gamma$ and if $j>i$, then $\rho_{vw^{p^{j-i}}}$ is a well-defined automorphism (see \cite[Prop. 5.5]{CorredorGutierrez}). In all these cases we call the induced automorphism of $G_\Gamma$ a \emph{dominated transvection}.

\begin{theorem}(\cite{CorredorGutierrez})
\label{GenSet}
	Let $G_\Gamma$ be a graph product of cyclic groups. The group $\Aut(G_\Gamma)$ is generated by labeled graph automorphisms, local automorphisms, partial conjugations and dominated transvections.
	In particular, $\Aut(G_\Gamma)$ is finitely generated.
\end{theorem}

Since we are now working with cyclic vertex groups, the center of $G_\Gamma$ is $G_\Delta$. Therefore we will denote the subgroup $T_\Delta$ of ${\rm Aut}(G_\Gamma)$ by $T_Z$ where the $Z$ should emphasize that for $f\in T_Z$ and $g\in G_\Gamma$ the element $g^{-1}f(g)$ is now always a central element.  

\begin{corollary}\label{TZfinite}
	Let $G_\Gamma$ be a graph product of cyclic groups. Let $\Delta^{fin}\subseteq \Delta$ be the induced subgraph generated by $\left\{v\in V(\Delta)\mid ord(G_v)<\infty\right\}$. We define a subgroup $T_Z^{fin}\subseteq T_Z$ as follows
	$$T_Z^{fin}:=\left\{f\in T_Z\mid\text{ for all }w_i\in V(\Gamma)-V(\Delta): w_i^{-1}f(w_i)\in G_{\Delta^{fin}}\right\}.$$
	
	The subgroup $T_Z^{fin}\subseteq\Aut(G_\Gamma)$ is normal and it is non-trivial if and only if $V(\Delta^{fin})\neq\emptyset$ and there exists a vertex $w\in V(\Gamma)-V(\Delta)$ such that $ord(G_w)=\infty$ or there exist a vertex $v\in V(\Delta^{fin})$ and $w\in V(\Gamma)-V(\Delta)$ such that $gcd(|G_v|, |G_w|)=p$. 
\end{corollary}
\begin{proof}
	By definition, an automorphism $f\in T_Z$ acts trivially on $G_\Delta$. Hence the order of $f$ is finite if and only if for every vertex $w\in V(\Gamma)-V(\Delta)$ the element $z_w$ defined by $z_w:=w^{-1}f(w)$ has finite order. Hence $T_Z^{fin}$ is nothing but the torsion part of the finitely generated abelian normal subgroup $T_Z$ of Aut$(G_\Gamma)$ and therefore normal itself. 
	
	Since it is generated by the dominated transvections $\rho_{wv^k}$ for vertices $w\in V(\Gamma)-V(\Delta)$ and $v\in V(\Delta)$ it is non-trivial if and only if at least one of these is well-defined. It is easy to see that this is the case if and only if the conditions in the corollary are satisfied.
\end{proof}

Concerning automatic continuity we will need the facts that torsion subgroups in $\Aut(G_\Gamma)$ are finite and $\Aut(G_\Gamma)$ does not have $\mathbb{Q}$ as a subgroup. The first property follows from the fact that ${\rm Out}(G_\Gamma):=\Aut(G_\Gamma)/\Inn(G_\Gamma)$ satisfies a stronger version of the Tits alternative, namely every subgroup in ${\rm Out}(G_\Gamma)$ is either full-sized or virtually polycyclic (see \cite[Corollary 4.4]{SaleSusse}). Hence, a torsion subgroup in ${\rm Out}(G_\Gamma)$ is virtually polycyclic. 

Recall that, by definition, a group $G$ is \emph{polycyclic} if there exists a descending chain of subgroups $\left\{1\right\}=G_n\subseteq G_{n-1}\ldots\subseteq G_1=G$ such that $G_{i+1}$ is normal in $G_i$ and the quotient $G_i/G_{i+1}$ is cyclic for $i=1,\ldots, n-1$.

\begin{corollary}
\label{TorsionFinite}
Let $\Aut(G_\Gamma)$ be the automorphism group of a graph product of cyclic groups and $H\subseteq\Aut(G_\Gamma)$ be a subgroup. If $H$ is a torsion group, then $H$ is finite.
\end{corollary}
\begin{proof}
   First we show that a torsion polycyclic group is finite. Let $T$ be a torsion polycyclic group. 
   Since $T$ is polycyclic, there exists a chain of subgroups	$\left\{1\right\}=T_n\subseteq T_{n-1}\ldots\subseteq T_1=T$ such that $T_{i+1}$ is normal in $T_i$ and $T_i/T_{i+1}$ is cyclic. 
	
    Now we prove that $T_{n-1}$ is finite. We know that $T_{n-1}/T_n$ is cyclic, so there exists $x_{n-1}\in T$ such that $\langle x_{n-1}\rangle =T_{n-1}$. By assumption $T$ is a torsion group, thus $x_{n-1}$ has finite order and therefore $T_{n-1}$ is finite. The same strategy shows that $T_{n-2}$ and then $T_j$ for $j=1,\ldots, n$ are finite groups. Hence $T_1=T$ is finite. 
	
	Moreover, a torsion virtually polycyclic group is also finite. More precisely, let $T$ be a torsion virtually polycyclic group. Then there exists a normal subgroup $N\trianglelefteq T$ such that $T/N$ is finite and $N$ is a torsion polycyclic group. We have
	$\left\{1\right\}\to N\to T\to T/N\to\left\{1\right\}$.
   Since $N$ is a torsion polycyclic group, this group is finite. Thus $N$ and $T/N$ are finite groups and therefore $T$ is also finite.
   
    Now we show that a torsion subgroup $H$ in $\Aut(G_\Gamma)$ is finite. Let $\pi\colon{\rm Aut}(G_\Gamma)\to{\rm Out}(G_\Gamma)$ be the canonical projection.
	Since $\pi(H)$ is not full-sized, we know that $\pi(H)$ is a torsion virtually polycyclic group which is finite by the above arguments.  So we have
		$\pi(H)\cong H/{H\cap{\rm Inn}(G_\Gamma)}$.
		The group $H\cap{\rm Inn}(G_\Gamma)$ is a torsion subgroup in ${\rm Inn}(G_\Gamma)\cong G_{\Gamma-\Delta}$. Lemma 3.6 in \cite{KramerVarghese} implies that this group is contained in a parabolic subgroup whose defining graph is a clique. Hence $H\cap{\rm Inn}(G_\Gamma)$ is finite and  therefore $H$ has to be finite too.
\end{proof}

The next lemma shows that the group of rational numbers can not be embedded in $\Aut(G_\Gamma)$. 
\begin{lemma}\label{noQ}
Let $G_\Gamma$ be a graph product of cyclic groups.
 The group ${\rm Aut}(G_\Gamma)$ does not have a subgroup isomorphic to $\mathbb{Q}$.
\end{lemma}	
\begin{proof}
 The group $G_\Gamma$ is linear, see \cite[Cor. 3.6]{HsuWise} and as such a group it is is residually finite \cite{Malcev}. Further, it was proven in \cite{Baumslag} that the automorphism group of a finitely generated residually finite group is residually finite. Hence $\Aut(G_\Gamma)$ is residually finite. Lemma 5.1 in  \cite{KeppelerMoellerVarghese} implies that a residually finite group does not have $\mathbb{Q}$ as a subgroup .
Hence ${\rm Aut}(G_\Gamma)$ does not have a subgroup isomorphic to $\mathbb{Q}$.
\end{proof}

\section{Proof of Theorem B}

We start by a characterisation of non full-sized normal subgroups in $\Aut(A_\Gamma)$ where both vertex sets $V(\Delta)=\left\{v_1\ldots, v_n\right\}$ and $V(\Gamma)-V(\Delta)=\left\{w_1,\ldots, w_m\right\}$ are non-empty. In that case $A_\Gamma$ has a non-trivial center and it is non abelian. In particular all groups, except the left and the right one, in the following exact sequence are non-trivial:
$$\left\{1\right\}\to T_Z\to\Aut(A_\Gamma)\overset{\pi=(\pi_1\times\pi_2)}{\rightarrow}\Aut(A_\Delta)\times \Aut(A_{\Gamma-\Delta})\to\left\{1\right\}.$$
Note also that $T_Z$ is generated by dominated transvections $\rho_{w_iv_j}$ for $i=1,\ldots, m$ and $j=1,\ldots, n$ whose order is infinite and any pair of these dominated transvections generate a free abelian group of rank $2$, hence $T_Z\cong \Z^{n\cdot m}$.

For the following proposition we recall the definition of the automorphism $\iota$ of a right-angled Artin group $A_\Gamma$. We define $\iota (v_i)=v_i^{-1}$ and $\iota (w_j)=w_j$ for $1\leq i\leq n$ and $1\leq j\leq m$ and extend this to a homomorphism.
\begin{proposition}
	\label{NormalArtin}
	Let $A_\Gamma$ be a right-angled Artin group. Let $V(\Delta)=\left\{v_1,\ldots, v_n\right\}$ and $V(\Gamma)-V(\Delta)=\left\{w_1,\ldots, w_m\right\}$ be both non-empty. Let $N\trianglelefteq\Aut(A_\Gamma)$ be a non-trivial normal subgroup. 
	
	If $N$ is non full-sized, then $N\subseteq T_Z\rtimes \left\{id, \iota\right\}\cong \Z^{n\cdot m}\rtimes\Z/2\Z$. Moreover, if  $N\subseteq T_Z$, then $N\cong \Z^l$ where $n\leq l\leq n\cdot m$.	
\end{proposition}
\begin{proof}
 	Let $N\trianglelefteq\Aut(A_\Gamma)$ be a non full-sized normal subgroup. Since $N$ is non full-sized, $\pi_1(N)\trianglelefteq\Aut(A_\Delta)\cong\GL_n(\Z)$ is non full-sized as well. Lemma \ref{F2} implies that $\pi_1(N)=\left\{id\right\}$ or $\pi_1(N)=\left\{id, \iota\right\}$. Further, $\pi_2(N)$ is also a normal subgroup of $\Aut(A_{\Gamma-\Delta})$ and is non full-sized. Hence by Proposition \ref{CenterlessAutG} the normal subgroup $\pi_2(N)$ is trivial.
	
	Now we consider the short exact sequence 
		$$\left\{1\right\}\to T_Z\to\Aut(A_\Gamma)\overset{\pi=(\pi_1\times\pi_2)}{\rightarrow}\Aut(A_\Delta)\times \Aut(A_{\Gamma-\Delta})\to\left\{1\right\}$$
	
	and restrict the map $\pi$ to $N$. We obtain two possibilities: 
	\begin{itemize}
	\item $1\to N\cap T_Z\to N\twoheadrightarrow \left\{id \right\}\times\left\{id\right\}\to\left\{1\right\}$
	In that case we have $N=N\cap T_Z$.
	
	\item $1\to N\cap T_Z\to N\twoheadrightarrow \left\{id, \iota\right\}\times\left\{id\right\}\to\left\{1\right\}$
	In that case we have $N\subseteq \langle T_Z, \iota\rangle\cong T_Z\rtimes\left\{id,\iota\right\}\cong\Z^{n\cdot m}\rtimes \Z/2\Z$.
	\end{itemize}

    It remains to prove that in the first case $N\cap T_Z$ is isomorphic to the free abelian group of rank $l$ where $n\leq l\leq n\cdot m$.
    
    Since we need to explicitly work with dominated transvections,
	let us first consider an easy example given by the dominated transvection $\rho_{w_1v_1}$.  Conjugating $\rho_{w_1v_1}$ with the labeled graph automorphism permuting the vertices of $\Delta$ as $\sigma$ in Sym$(n)$ yields the dominated transvection $\rho_{w_1v_{\sigma(1)}}$. Similarly if there exists a labeled graph automorphism permuting (some of) the vertices in $\Gamma-\Delta$ stemming from a $\tau\in {\rm Sym}(m)$, we obtain the dominated transvection $\rho_{w_{\tau(1)}v_1}$ via conjugation. So in this case the normal closure of $\rho_{w_1v_1}$ is isomorphic to $\mathbb{Z}^l$ for some $n\leq l\leq n\cdot m$, where the upper bound comes from the observation that $T_Z\cong \mathbb{Z}^{n\cdot m}$.  Thus  $N\cap T_Z$ is always isomorphic to some $\Z^k$ for some $k\leq n\cdot m$, so all that remains to show is that $k\geq n=|V(\Delta)|$.
	
	We know that given $w_i\in V(\Gamma)-V(\Delta)$, the automorphism $f\in N$ must satisfy $f(w_i)=w_iz_i$ for some $z_i\in A_\Delta$. This means $f(w_i)=w_iv_1^{k_1}\cdot ...\cdot v_n^{k_n}$ ($k_j\in \Z$) and we obtain a homomorphism $h_i\colon N\to \mathbb{Z}^n$ by collecting the exponents, i.e. $h_i(f):=(k_1,...,k_n)$. Note that this is a homomorphism since all elements of $T_Z$ are products of dominated transvections. 
	
	Suppose that $f$ is not the identity homomorphism. Then at least one of the maps $h_i$ is non-trivial. Let $\sigma\in{\rm Sym}(n)$ denote a permutation and $\alpha$ denote the corresponding labeled graph automorphism permuting the vertices of $A_\Delta$. Due to completeness of $\Delta$ there always exists such a labeled graph automorphism. Then $h_i(\alpha\circ f\circ \alpha^{-1})=(k_{\sigma(1)},...,k_{\sigma(n)})$. If $\{(k_{\sigma(1)},...,k_{\sigma(n)})\mid \sigma\in{\rm Sym}(n)\}$ is a generating system for $\mathbb{Z}^n$ then we are done. If not that means that $k_1=k_2=...=k_n\neq 0$ since $h_i(f)\neq 0$ and $h_i(\alpha\circ f\circ \alpha^{-1})=(k_{\sigma(1)},...,k_{\sigma(n)})$. Now we conjugate $f$ by the dominated transvection $\rho_{v_1v_2}$. Note that there are always at least two vertices in $\Delta$ because $\langle k_1\rangle$ is isomorphic to $\Z$ for every $k_1\neq 0$. Evaluating the conjugation of $f$ by $\rho_{v_1v_2}$ on $w_1$ gives us $\rho_{v_1v_2}\circ f\circ \rho_{v_1v_2}^{-1}(w_i)=\rho_{v_1v_2}\circ f(w_i)=w_iz_iv_2^{k_1}$. Thus we obtain $h_i(\rho_{v_1v_2}\circ f\circ \rho_{v_1v_2}^{-1})=(k_1,k_2+ k_1,\ldots ,k_n)$. But now not all entries are equal, thus the group generated by all permutations is isomorphic to $\mathbb{Z}^n$ which completes the proof.
\end{proof}

Now we have all ingredients to prove Theorem B.

\begin{proof}[Proof of Theorem B]
	Let $N\trianglelefteq\Aut(A_\Gamma)$ be a non full-sized normal subgroup. Since algebraic properties of $\Aut(A_\Gamma)$ depend on the combinatorial structure of the defining graph $\Gamma$ we have to consider several cases:
	
	First we investigate the case where $V(\Gamma)-V(\Delta)=\emptyset$. In that case  $\Aut(A_\Gamma)\cong\GL_n(\Z)$ and Lemma \ref{GLnoF2} implies that $N$ is trivial or isomorphic to the center of $\GL_n(\Z)$ which is equal to $\left\{I, -I \right\}$.
	
	The next case deals with the situation where $V(\Delta)$ is empty. Proposition \ref{CenterlessAutG} implies directly that $N$ is trivial.
	
	Now we consider the last case: assume that
	$V(\Delta)\neq\emptyset$ and $V(\Gamma)-V(\Delta)\neq\emptyset$ and let $N$ be maximal. By Proposition \ref{NormalArtin} we have  $N\subseteq T_Z\rtimes\left\{id, \iota\right\}$. Using the generating set of $\Aut(A_\Gamma)$ from Theorem \ref{GenSet} it is straightforward to prove that $T_Z\rtimes\left\{id, \iota\right\}$ is normal in $\Aut(A_\Gamma)$. 
	
	If $N$ is a minimal, non-trivial normal subgroup in $\Aut(A_\Gamma)$, then Proposition \ref{NormalArtin} implies that $N\subseteq T_Z$ and $N\cong \Z^l$ where $n\leq l\leq n\cdot m$. More precisely, first note that $\{id,\iota\}$ is not normal in $\Aut(A_\Gamma)$ since $\rho_{w_1v_1}\circ \iota\circ \rho_{w_1v_1}^{-1}\notin \left\{id, \iota\right\}$, thus $N\cap T_Z$ is a non-trivial normal subgroup. Hence, if $N\nsubseteq T_Z$, then $N\cap T_Z$ is a smaller non-trivial normal subgroup of $\Aut(A_\Gamma)$. This finishes the proof.
\end{proof}

In particular, Theorem B implies Corollary C from the introduction. 

\section{Proof of Theorem G}

Before we move on to the proof of Theorem G let us recall a result about the structure of the automorphism group of a finitely generated abelian group. If $\Gamma$ is a complete graph, then $G_\Gamma$ is a finitely generated abelian group. Thus there exist $n\in\mathbb{N}$ and a finite abelian group $T$ such that $G_\Gamma\cong\Z^n\times T$. Since $T$ is a characteristic subgroup we always have two group homomorphisms:
$$\rho_1\colon\Aut(\Z^n\times T)\to\Aut(\Z^n)\cong\GL_n(\Z)\text{ and
}\rho_2\colon\Aut(\Z^n\times T)\to\Aut(T).$$
Note that both maps are  surjective. Further, by construction, the map $\rho=(\rho_1\times\rho_2)$ is also surjective. Hence, the  sequence
$$\left\{1\right\}\to\ker(\rho)\to\Aut(\Z^n\times T)\overset{\rho}{\rightarrow}\GL_n(\Z)\times\Aut(T)\to\left\{1\right\}$$ is exact. Additionally, it was proven in \cite[Lemma 2.16]{SaleSusse} that $\ker(\rho)\cong T^n$.

Thus if $\Gamma=\Delta$  and has at least one infinite cyclic vertex group and at least one finite vertex group, then $\ker(\rho)$ is a non-trivial finite normal subgroup in $\Aut(G_\Gamma)$.

Using similar methods as in the proof of Theorem B, we prove Theorem G.

\begin{proof}[Proof of Theorem G]
It follows directly from Proposition \ref{CenterlessAutG} that if $V(\Delta)=\emptyset$, then any finite normal subgroup in $\Aut(G_\Gamma)$ is trivial.

Let us now assume that $V(\Gamma)-V(\Delta)$ is empty. The goal is to prove that  $\Aut(G_\Gamma)$ does not have non-trivial finite normal subgroups if and only if $V(\Gamma)=\left\{v_1\right\}$ and $ord(G_{v_1})=2$.  
The if statement is clear since $\Aut(\Z/2\Z)=\{id\}$. For the only if statement we differentiate between two cases.
First assume that all vertex groups are infinite. Then $\Aut(G_\Gamma)\cong \GL_n(\Z)$ and hence has a non-trivial finite normal subgroup, the center, which is isomorphic to $\Z/2\Z$. In the other case there either exists at least one infinite vertex group and at least one finite vertex group, in which case $\ker(\rho)$ is a non-trivial finite normal subgroup or all vertex groups are finite. If all vertex groups are finite, then $\Aut(G_\Gamma)$ is itself finite and this group is non-trivial if and only if $\Gamma\neq\left\{v_1\right\}$ where $ord(v_1)=2$.  
    
Now assume that $V(\Delta)=\left\{v_1,\ldots, v_n\right\}$ is non-empty and $V(\Gamma)-V(\Delta)=\left\{w_1,\ldots, w_m\right\}$ is non-empty. We have the splitting short exact sequence:
$$\left\{1\right\}\to T_Z\to {\rm Aut}(G_\Gamma)\overset{\pi=(\pi_1\times\pi_2)}{\rightarrow} {\rm Aut}(G_\Delta)\times {\rm Aut}(G_{\Gamma-\Delta})\to\left\{1\right\}$$
Assume that all vertex groups in $V(\Delta)$ are infinite cyclic and that there is an infinite vertex group $G_w$ in $\Gamma-\Delta$. Let $N\trianglelefteq {\rm Aut}(G_\Gamma)$ denote a finite normal subgroup. We know that $T_Z$ is a finitely generated abelian group. If $G_\Delta$ is torsion free, it follows immediately that $T_Z$ is torsion free, hence $T_Z\cap N=\{1\}$. Furthermore due to above results we know that $\pi_1(N)\subseteq\left\{id, \iota\right\}$ and $\pi_2(N)=\{id\}$. Hence $N$ can only be isomorphic to $\Z/2\Z$ and be generated by an element $f$ which inverts all vertices in $\Delta$ and maps every vertex $w\in V(\Gamma-\Delta)$ to $wz_w$ for some $z_w\in G_\Delta$. But such an $f$ does not commute with the dominated transvection $\rho_{wv}$, since $\rho_{wv}\circ f (w)=wz_wv$ and $f\circ \rho_{wv}(w)=wz_wv^{-1}$. Hence $f$ has to be trivial.

Note that if all vertex groups in $V(\Gamma)-V(\Delta)$ are finite, then Lemma \ref{CharacteristicBothFactors} implies that $\Aut(G_\Gamma)\cong\Aut(G_\Delta)\times\Aut\left(G_{\Gamma-\Delta}\right)$ and hence the subgroup $\left\{id, \iota\right\}$ is normal in $\Aut(G_\Gamma)$. We proved that if all vertex groups in $V(\Delta)$ are infinite cyclic, then $\Aut(G_\Gamma)$ does not have non-trivial finite normal subgroups if and only if there exists an infinite vertex group $G_{w_j}$.

If there exists $v_i\in V(\Delta)$ such that $G_{v_i}$ is finite, then we have to differ two cases and prove the following statements:
		\begin{enumerate}
		    \item[(a)] If $n=1$ and $ord(G_{v_1})=2$, then $\Aut(G_\Gamma)$ does not have non-trivial finite normal subgroups if and only if  $ord(G_{w_j})<\infty$  and $2\nmid ord(G_{w_j})$ for $j=1,\ldots, m$.  
		    \item[(b)] If $n\geq 2$ or $n=1$ and $ord(G_{v_1})\neq 2$, then $\Aut(G_\Gamma)$ does have non-trivial finite normal subgroups.
		\end{enumerate} 
		
		For the only if statement in (a) simply note that $T_Z^{fin}$ is a finite normal subgroup of Aut$(G_\Gamma)$ and it is non-trivial if one of the conditions is satisfied, see Corollary \ref{TZfinite}. If those conditions are not satisfied it is easy to check that $G_\Delta$ and $G_{\Gamma-\Delta}$ are characteristic subgroups and hence by Lemma \ref{CharacteristicBothFactors} we have Aut$(G_\Gamma)={\rm Aut}(G_\Delta)\times {\rm Aut}(G_{\Gamma-\Delta})$. But Aut$(G_\Delta)=\{1\}$ and $\Aut(G_{\Gamma-\Delta})$ does not have non-trivial finite normal subgroups  which finishes off (a).

Finally consider case (b). If there exists at least one dominated transvection $\rho_{wv}$ for a vertex $w\in V(\Gamma-\Delta)$ and a vertex $v\in V(\Gamma)$ with $ord(v)<\infty$, then $T_Z^{fin}$ is a non-trivial, finite normal subgroup of Aut$(G_\Gamma)$. If no such dominated transvection exists, then Aut$(G_\Gamma)={\rm Aut}(G_\Delta)\times {\rm Aut}(G_{\Gamma-\Delta})$. We know that a non-trivial finite normal subgroup in Aut$(G_\Delta)$ exists which in turn needs to be normal in Aut$(G_\Gamma)$ due to the direct product structure which proves (b).
\end{proof}

\section{The center of the automorphism group of a graph product}
In this chapter we study the center of the automorphism group of a graph product of cyclic groups. First we discuss two known but useful results and then show that the center is always finite. We obtain conditions when the automorphism group is centerless and in particular prove Corollary D. Afterwards we explicitly calculate the center of the automorphism group if the defining graph is complete. Then we study the case where the defining graph is a join of one vertex with the rest and also calculate the center. This is then applied to right-angled Coxeter groups and an explicit example is discussed where the center of the automorphism group is isomorphic to $(\Z/2\Z)^{12}$.

We start this chapter by recalling a result regarding the center of the automorphism group that is well known. Since we could not find a good reference, we include the proof of it here. 
\begin{lemma}
\label{center}
    Let $G$ be a group and $\Aut(G)$ the automorphism group. If $f\in Z(\Aut(G))$, then for $g\in G$ there exists $z_g\in Z(G)$ such that $f(g)=gz_g$.
    
    In particular, if $G$ is centerless, then the center of $\Aut(G)$ is trivial.
\end{lemma}
\begin{proof}
    For $g\in G$ we denote by $\gamma_{g}\in\Inn(G)$ the conjugation by $g$. Let us consider the  equation
    $\gamma_{g}\circ f\circ \gamma_{g^{-1}}=f$ that holds, since $f$ commutes with all automorphisms of $G$.
    Thus for $x\in G$ we obtain 
    $\gamma_{g}\circ f\circ \gamma_{g^{-1}}(x)=f(x)$ which is equivalent to 
    $gf(g)^{-1}f(x)f(g)g^{-1}=f(x)$.
    This shows that $gf(g)^{-1}$ commutes with $f(x)$ and since $x$ is an arbitrary element in $G$ it follows that $gf(g)^{-1}\in Z(f(G))=Z(G)$. Hence, there exists $z_g\in Z(G)$ such that $f(g)=gz_g.$
    
    Furthermore, if $Z(G)=\left\{1\right\}$, then $z_g=1$ for all $g\in G$ and therefore $f$ is equal to identity. Thus, if $G$ is centerless, then the center of $\Aut(G)$ is trivial.
\end{proof}

The center of the automorphism group of a free abelian group and of a finite abelian group have been explicitly calculated, for completeness we state the results here. 
\begin{proposition}
\label{centerFreeabelianAndFinite}
Let $G_\Gamma$ be a graph product of cyclic groups. Assume that $\Gamma$ is a clique. 
	\begin{enumerate}
		\item If all vertex groups are infinite cyclic, then $$Z(\Aut(G_\Gamma)\cong Z({\rm GL}_n(\Z))=\left\{I, -I\right\}.$$
		\item If all vertex groups are finite, then $$Z(\Aut(G_\Gamma))\cong\prod_{p\in J}(\Z/p^{l_p}\Z)^*$$
		where $J:=\left\{p\in\mathbb{P}\mid \text{ there exists a vertex }v\in V(\Gamma)\text{ with } p|ord(v)\right\}$ and $p^{l_p}$ is the maximal order of a vertex group in $\Gamma$ with $p$ power order.
		\end{enumerate}
\end{proposition}
The result of Proposition \ref{centerFreeabelianAndFinite}(i) is well known and the result of Proposition \ref{centerFreeabelianAndFinite}(ii) was shown in \cite[Proposition 3.2]{HanZhou}.

Before we move on to the concrete calculation of the center of the automorphism group of a graph product of cyclic groups, we show that the center is always finite. 
\begin{lemma}
	\label{CenterFiniteDelta}
Let $G_\Gamma$ be a graph product of cyclic groups. The center of $\Aut(G_\Gamma)$ is finite.	
\end{lemma}	
\begin{proof}
First we note that  if there is no vertex $v\in V(\Gamma)$ such that $st(v)=V(\Gamma)$,  then $G_\Gamma$ is centerless and therefore by Lemma \ref{center} the group $\Aut(G_\Gamma)$ is also centerless.

Now, let us assume that $\Gamma$ is a complete graph, then $G_\Gamma$ is a finitely generated abelian group and therefore there exists $n\in\mathbb{N}$ and a finite abelian group $T$ such that $G_\Gamma\cong\Z^n\times T$. We consider the exact sequence
$$\left\{1\right\}\to\ker(\rho)\to\Aut(\Z^n\times T)\to\GL_n(\Z)\times\Aut(T)\to \left\{1\right\}$$
where $\rho=(\rho_1\times \rho_2)$ and $\ker(\rho)\cong T^n$.

Consider the restriction of this exact sequence to $Z(\Aut(\Z^n\times T))$. Note, that $\rho_1(Z(\Aut(\Z^n\times T)))\subseteq Z(\GL_n(\Z))=\left\{I, -I\right\}$. 
We obtain
$$ \ker(\rho)\cap Z(\Aut(\Z^n\times T))\to Z(\Aut(\Z^n\times T))\to \rho_1(Z(\Aut(\Z^n\times T)))\times \rho_2(Z(\Aut(\Z^n\times T)))$$
The group $\ker(\rho)\cap Z(\Aut(\Z^n\times T))\subseteq T^n$ is finite and the group $\rho_1(Z(\Aut(\Z^n\times T)))\times \rho_2(Z(\Aut(\Z^n\times T)))\subseteq \left\{I, -I\right\}\times\Aut(T)$ is finite too, hence $Z(\Aut(\Z^n\times T))$ is finite. 

In the last case we assume that  $V(\Delta)=\left\{v_1,\ldots, v_n\right\}$ and $V(\Gamma)-V(\Delta)=\left\{w_1,\ldots, w_m\right\}$ for $n,m\geq 1$. Let $f\in Z(\Aut(G_\Gamma))$. Our goal is to show that $f$ has finite order.

Since $G_\Delta$ is characteristic, we get a group homomorphism $\pi_1\colon\Aut(G_\Gamma)\to\Aut(G_\Delta)$ that is surjective. Note, that $\pi_1(Z(\Aut(G_\Gamma)))\subseteq Z(\Aut(G_\Delta))$.

Hence we know that $f_{|G_\Delta}$ has finite order, since $f_{|G_\Delta}\in Z(\Aut(G_\Delta))$. By Lemma \ref{center} it follows that for $w_j$ there exists $z_j\in G_\Delta$ such that $f(w_j)=w_jz_j$. Thus $f$ has finite order if and only if all $z_j$ are of finite order. Assume for contradiction, that one of the $z_i$'s has infinite order. Then we decompose $z_i=x_1x_2\ldots x_ky_1\ldots y_l$ where $x_o,y_j\in V(\Delta)$ and $ord(x_o)=\infty$ for all $o\in\{1,...,k\}$ and $ord(y_j)<\infty$ for all $j\in \{1,...,l\}$. Let $h\in \Aut(G_\Gamma)$ be the automorphism defined as $h(x_1)=x_1^{-1}$ and identity elsewhere.

Then $f$ does not commute with $h$. More precisely, we have
$$f\circ h(w_i)=f(w_i)=w_ix_1x_2\ldots x_ky_1\ldots y_l$$
and
$$h\circ f(w_i)=h(w_ix_1x_2\ldots x_ky_1\ldots y_l)=w_ix^{-1}_1x_2\ldots x_ky_1\ldots y_l$$

Hence $f$ has finite order and therefore $Z(\Aut(G_\Gamma))$ is an abelian torsion group. By Lemma \ref{TorsionFinite} we know that any torsion subgroup in $\Aut(G_\Gamma)$ is finite, thus $Z(\Aut(G_\Gamma))$ is finite.
\end{proof}

The above lemma together with Theorem G implies 
\begin{corollary}
Let $G_\Gamma$ be a graph product of cyclic groups. If 
\begin{enumerate}
    \item $V(\Delta)$ is empty or
    \item $V(\Delta)=\left\{v_1\right\}$, $ord(v_1)=2$ and $V(\Gamma)-V(\Delta)$ is non-empty and all vertex groups in this set are finite and $2\nmid ord(w_j)$ for $j=1,\ldots, m$ or
    \item $V(\Delta)$ is non-empty and $V(\Gamma)-V(\Delta)$ is non-empty and all vertex groups in $\Delta$ are infinite cyclic and there exists an infinite vertex group in $V(\Gamma)-V(\Delta)$,
\end{enumerate}
then $\Aut(G_\Gamma)$ is centerless.
\end{corollary}
In particular, we have proven Corollary D from the introduction.

\subsection{The center of the automorphism group of an infinite finitely generated abelian group with torsion}
Let us now focus on the calculation of the center of the automorphism group of an infinite finitely generated abelian group with torsion elements $G\cong\Z^n\times T$ where $n\geq 1$ and $T$ is a non-trivial finite abelian group. We denote by $\left\{x_1,\ldots, x_n\right\}$ a generating set of the free abelian part of $G$ and by $\left\{y_1,\ldots, y_m\right\}$ a generating set of the torsion part of $G$ where each $y_j$ has a prime power order. Note that we can write $G$ as a graph product $G_\Delta$ where $V(\Delta)=\left\{x_1,\ldots, x_n, y_1,\ldots, y_m\right\}$.

\begin{lemma}\label{iotainj}
Let $G=\Z^n\times T$ be a finitely generated infinite abelian group where $T$ is a finite non-trivial abelian group. Let $\iota\in\Aut(G)$ be the automorphism induced by $x_i\mapsto x_i^{-1}$ and $y_j\mapsto y_j$ for all $i=1,\ldots, n$, $j=1,\ldots, m$.
We have $\iota\in Z(\Aut(G))$ if and only if $T$ is a product of $\Z/2\Z$'s.
\end{lemma}
\begin{proof}
We write $G$ as a graph product of a complete graph $\Delta$, where $x_1,\ldots,x_n$ are the infinite order vertices and $y_1,\ldots,y_m$ the finite order ones, hence $G\cong G_\Delta$. For the only if part we assume there exists a vertex $y_j\in V(\Delta)$ with $ord(y_j)\neq 2$. Then the dominated transvection $\rho:=\rho_{x_1y_j}\in \Aut(G_\Delta)$ does not commute with $\iota$, since:
$$\iota\circ \rho(x_1)=x_1^{-1}y_j \quad \text{ and } \quad \rho\circ \iota (x_1)=x_1^{-1}y_j^{-1}$$
and the equality $y_j=y_j^{-1}$ holds if and only if $ord(y_j)=2$.

For the if part, the above calculation shows that $\iota$ commutes with all dominated transvections of the form $\rho_{x_iy_l}$ for $i=1,\ldots,n$ and $l=1,\ldots,m$. It is an easy exercise to show that $\iota$ commutes with local automorphisms and graph automorphisms as well.
\end{proof}

Given a finitely generated abelian group $G\cong\Z^n\times T$, we always have a surjective group homomorphism  $\rho:=(\rho_1\times\rho_2)\colon\Aut(G)\to \Aut(\Z^n)\times\Aut(T)$ induced by the fact that $T$ is a characteristic subgroup. The next lemma describes when the restriction of $\rho$ to the center of $\Aut(G)$ is injective. 

\begin{lemma}\label{rhoinj}
Let $G\cong\Z^n\times T$ be a finitely generated infinite abelian group where $T$ is a finite group. Let $\rho\colon\Aut(G)\to\Aut(\Z^n)\times\Aut(T)$ be as above.

The restriction $\rho_{\mid Z(\Aut(G))}$ is injective if and only if there exist no non-trivial central transvections, which is equivalent to the following condition: Write $T=\langle y_1\rangle \times ... \times \langle y_m\rangle$ and consider the maximum order of $y_i$, that is equal to a power of $2$. Call this order $k$. Then there exist non-trivial central transvections if and only if $n=1$ and $|\left\{y_i\mid ord(y_i)=k\right\}|=1$. 

If the kernel of the restriction is non-trivial, then it is isomorphic to $\Z/2\Z$.
\end{lemma}
\begin{proof}
The kernel of $\rho$ is isomorphic to $T^n$ as discussed in the proof of Lemma \ref{CenterFiniteDelta}. 

The fact that $\rho_{\mid Z(\Aut(G))}$ is injective if and only if there exist no non-trivial central transvections follows from the short exact sequence used in the proof of Lemma \ref{CenterFiniteDelta}.

For the ``only if'' part first assume that in the torsion part, no vertex order is a power of $2$. Then no elements of order $2$ exist and hence by the calculation in the previous lemma, no dominated transvection $\rho_{x_iy_j}$ for $i=1,\ldots,n$ and $j=1,\ldots,m$ commutes with $\iota$. Now suppose there are multiple vertices with order $k$. Without loss of generality we can assume that $ord(y_1)=k=ord(y_2)$. Towards a contradiction assume there exists a non-trivial element $f$ in the kernel of $\rho_{|Z(\Aut(G))}$. We know that $f(x_i)=x_iz_{i}$ and that the order of $z_i$ is $1$ or $2$ for every $i$. Without loss of generality $z_1\neq 1$. Write $z_1=v_1^{l_1}...v_o^{l_o}$, where the order of all $v_j$ is a power of $2$ and $l_1,...,l_o \in \Z-\{0\}$. 

We first show that $ord(v_i)\neq k$. Assume for the contrary that $ord(v_1)=k$. Then all vertices with order $k$ need to appear as factors in $z_1$, since they can be permuted by a labeled graph automorphism and $f$ needs to commute with these. So we can assume that $ord(v_2)=k$ as well. Hence the dominated transvection $\rho_{v_1v_2}$ is well defined. But $f$ does not commute with $\rho_{v_1v_2}$ since $\rho_{v_1v_2}\circ f (x_1)=x_1v_1^{l_1}v_3^{l_{3}}v_4^{l_4}...v_o^{l_o}$ and $f\circ \rho_{v_1v_2}(x_1)=x_1v_1^{l_{1}}v_2^{l_2}...v_o^{l_o}$. 

Now we show that the other orders cannot appear either. Suppose that the order of $v_1$ is $2^a$ and that $ord(y_1)=k=2^b$ with $b>a$. Then the dominated transvection $\rho_{v_1y_1^{2^{b-a}}}$ is well-defined and does not commute with $f$ since
$$f\circ \rho_{v_1y_1^{2^{b-a}}}(x_1)=x_1z_1\quad\text{and}\quad \rho_{v_1y_1^{2^{b-a}}}\circ f (x_1)=x_1z_1y_1^{2^{b-1}}$$
and $y_1$ has order $2^b$, so $y_1^{2^{b-1}}$ is non-trivial. Hence $z_1=1$, a contradiction.

The final case is that there is exactly one vertex of order $k$ but $n>1$. Then by the previous case, the only possible central element would be of the form $f(x_i)=x_iy^{k/2}$ where $y$ is the unique vertex of order $k$. Now there exists the dominated transvection $\rho_{x_1x_2}$ and this does not commute with f since $f\circ \rho_{x_1x_2}(x_1)=x_1x_2$ but $\rho_{x_1x_2}\circ f (x_1)=x_1x_2y^{k/2}$ finishing the proof of the ``only if'' direction.

For the if part assume that the order of $y_1$ is $k$. Then the dominated transvection $\rho_{x_1y_1^{k/2}}$ is central and hence in the kernel of $\rho_{\mid Z(\Aut(G))}$. This can be seen by doing essentially the same calculations as in the previous lemma. The calculations of the other direction show that this dominated transvection is the only possible central element of the kernel.
\end{proof}

Before calculating the center of the automorphism group of a graph product $G_\Delta$ in full generality let us discuss the definition of one very important automorphism. 

If there exists at least one vertex whose order is a power of $2$, then  let $k$ denote the maximum order of the $y_j$ such that $ord(y_j)$ is a power of $2$. Write $\{z_1,...,z_l\}:=\{y_j\mid ord(y_j)=k\}$. Then define $\alpha\in\Aut(G_\Delta)$ by $\alpha(x_i)=x_iz_1...z_l$ and $\alpha(y_j)=y_j$.

First, we investigate the center of $\Aut(G_\Delta)$ where $ord(y_i)=2$ for $i=1,\ldots, m$.
\begin{corollary}
\label{ZZ2}
   Let $G_\Delta\cong \Z^n\times(\Z/2\Z)^m$ where $n,m \geq 1$. The center of $\Aut(G_\Delta)$ is isomorphic to $\langle \iota, \alpha\rangle\cong\Z/2\Z\times\Z/2\Z$ if and only if $n=1$ and $m=1$. In all other cases the center of $\Aut(G_\Delta)$ is isomorphic to $\left\{1,\iota\right\}$. 
\end{corollary}
\begin{proof}
    For the ``if'' statement, simply note that $\Aut(G_\Delta)\cong\langle\iota,\alpha\rangle\cong \Z/2\Z\times \Z/2\Z$ by Theorem \ref{GenSet}, since the only generators are the local automorphism $\iota$ and one dominated transvection mapping the generator of $\Z$ to the product of itself with the generator of $\Z/2\Z$.
    
    In the other cases we know by Lemma \ref{rhoinj}, that $\rho_{\mid Z(\Aut(G_\Delta))}$ is injective. Thus we have
    $$\rho_{|\Aut(G_\Delta)}\colon Z(\Aut(G_\Delta))\hookrightarrow Z(\Aut(\Z^n))\times Z(\Aut(\Z/2\Z)^m)=\left\{1,\iota\right\}\times\left\{1\right\}.$$
    Hence the center of $\Aut(G_\Delta)$ has at most $2$ elements. Moreover it is easy to check that $\iota$ is a central element using the generating set from Theorem \ref{GenSet}. 
\end{proof}

Now we move on to the general case.
\begin{proposition}
\label{centerDelta}
	Let $G_\Delta$ be a graph product of cyclic groups where $\Delta$ is a clique. We write $V(\Delta)=\left\{x_1,\ldots, x_n,y_1,\ldots, y_m\right\}$ where the order of $x_i$ for $i=1,\ldots, n$ is infinite and each order of $y_j$ for $j=1,\ldots, m$ is a prime power. 
	If $n, m\geq 1$, then the center of $\Aut(G_\Gamma)$ is a subgroup of $\langle \iota, \alpha\rangle\cong \Z/2\Z \times \Z/2\Z $.

      Furthermore let $k$ denote the maximal order of a vertex group which is also a power of $2$. Then the center is
		\begin{enumerate}
		    \item isomorphic to $\{1,\alpha\}$ if there exists exactly one vertex of infinite order and precisely one vertex of order $k$ and $k\neq 2$ or precisely one vertex of order $2$, $k=2$ and at least one vertex of finite order not equal to $2$.
		
		    \item isomorphic to $\langle\iota, \alpha\rangle$ if there is precisely one vertex order of finite order and this order is $2$ and precisely one vertex of infinite order.
		     \item isomorphic to $\{1,\iota\}$ if all finite vertex groups have order $2$ and there are more than $2$ vertices.
        \item trivial else.
		\end{enumerate}
\end{proposition}

\begin{proof}
We start by invoking Lemma \ref{rhoinj} which tells us that $\rho_{\mid Z(\Aut(G_\Delta))}$ is injective in cases (iii) and (iv) and else not injective (this follows by carefully checking all possibilities).
In the cases where $\rho_{\mid Z(\Aut(G_\Delta))}$ is injective we know that the center is a subgroup of $\{1,\iota\}\times Z(\Aut(\langle y_1,\ldots, y_m\rangle)$. However no automorphism of $Z(\Aut(\langle y_1,\ldots, y_m\rangle))$ can appear in $Z(\Aut(G_\Delta))$ since these appear as local automorphisms by Proposition \ref{centerFreeabelianAndFinite}(ii) and $\langle y_1,\ldots, y_m\rangle$ is characteristic. These do not commute with the product of the dominated transvections $\rho_{x_1y_1}\rho_{x_1y_2}...\rho_{x_1y_m}$. More precisely let $\rho_1\colon \Aut(G_\Delta)\to \Aut(\Z^n)$ and $\rho_2\colon \Aut(G_\Delta)\to \Aut(\langle y_1,\ldots, y_m\rangle)$ denote the respective projections. Since the center of $\Aut(G_\Delta)$ is mapped to the center of the direct product, we know a central element $f\in \Aut(G_\Delta)$ has to map an infinite order vertex $x$ to $xt$ or $x^{-1}t$ for some torsion element $t\in \langle y_1,\ldots, y_m\rangle$ and $f(y)=y^{k_y}$ for every vertex in $\left\{y_1,\ldots, y_m\right\}$. If the $k_y$ is not equal to $1$, then this $f$ does not commute with the dominated transvection $\rho_{xy}$. Therefore we conclude that a central element $f$ acts as the identity on $\langle y_1,\ldots, y_m\rangle$. Hence in cases (iii) and (iv) the only possible non-trivial central element is $\iota$, so we use Lemma \ref{iotainj} to finish these two cases.

Corollary \ref{ZZ2} implies the case (ii).

The remaining case is (i). It is clear by Lemma \ref{rhoinj}, that we have the following inclusion $Z(\Aut(G_\Delta))\subseteq \langle \alpha, \iota , \Aut(\langle y_1,\ldots, y_m\rangle)\rangle$. By the above argument, we can see that a central element once again acts trivially on $\langle y_1,\ldots, y_m\rangle$, hence $Z(\Aut(G_\Delta))\subseteq \langle\iota,\alpha\rangle $ and since $\iota$ is not central by Lemma \ref{iotainj} and $\alpha$ is by the same argument as in the proof of Lemma \ref{rhoinj}, we obtain the desired statement.
\end{proof}

Let us discuss some examples. We denote the graphs in Figure 6 from left to the right by $\Delta_1, \Delta_2, \Delta_3$ and $\Delta_4$. Proposition \ref{centerDelta}(iii) implies that the center of $\Aut(G_{\Delta_1})$ is generated by $\iota$, thus $Z(\Aut(G_{\Delta_1}))\cong\Z/2\Z$. Further, by Proposition \ref{centerDelta}(iv) we know that $Z(\Aut(G_{\Delta_2}))$ is trivial. The center of $\Aut(G_{\Delta_3})$ is equal to $\left\{id, \alpha\right\}$  by Proposition \ref{centerDelta}(i). We have $Z(\Aut(G_{\Delta_4}))=\left\{id\right\}$ by Proposition \ref{centerDelta}(iv).
 
\begin{figure}[h]
	\begin{center}
		\begin{tikzpicture}
			\draw[fill=black]  (0,0) circle (3pt);
			\node at (0,-0.4) {$\Z$};
			\draw[fill=black]  (2,0) circle (3pt);
			\node at (2,-0.4) {$\Z$};	
			\draw[fill=black]  (2,2) circle (3pt);
			\node at (2,2.4) {$\Z$};
			\draw[fill=black]  (0,2) circle (3pt);
			\node at (0,2.4) {$\Z$};	
			\draw (0,0)--(2,0);	
			\draw (2,0)--(2,2);	
			\draw (2,2)--(0,2);
			\draw (0,0)--(0,2);
			\draw (0,0)--(2,2);	
			\draw (0,2)--(2,0);	
			
			\draw[fill=black]  (4,0) circle (3pt);
			\node at (4,-0.4) {$\Z$};
			\draw[fill=black]  (6,0) circle (3pt);
			\node at (6,-0.4) {$\Z$};	
			\draw[fill=black]  (6,2) circle (3pt);
			\node at (6,2.4) {$\Z/3\Z$};
			\draw[fill=black]  (4,2) circle (3pt);
			\node at (4,2.4) {$\Z/5\Z$};	
			\draw (4,0)--(6,0);	
			\draw (6,0)--(6,2);	
			\draw (4,2)--(6,2);
			\draw (4,0)--(4,2);
			\draw (4,0)--(6,2);	
			\draw (4,2)--(6,0);	
			
			\draw[fill=black]  (8,0) circle (3pt);
			\node at (8,-0.4) {$\Z$};
			\draw[fill=black]  (10,0) circle (3pt);
			\node at (10,-0.4) {$\Z/2\Z$};	
			\draw[fill=black]  (10,2) circle (3pt);
			\node at (10,2.4) {$\Z/3\Z$};
			\draw[fill=black]  (8,2) circle (3pt);
			\node at (8,2.4) {$\Z/5\Z$};	
			\draw (8,0)--(10,0);	
			\draw (10,0)--(10,2);	
			\draw (10,2)--(8,2);
			\draw (8,0)--(8,2);
			\draw (8,0)--(10,2);	
			\draw (8,2)--(10,0);	
			
			\draw[fill=black]  (12,0) circle (3pt);
			\node at (12,-0.4) {$\Z/8\Z$};
			\draw[fill=black]  (14,0) circle (3pt);
			\node at (14,-0.4) {$\Z/16\Z$};	
			\draw[fill=black]  (14,2) circle (3pt);
			\node at (14,2.4) {$\Z/16\Z$};
			\draw[fill=black]  (12,2) circle (3pt);
			\node at (12,2.4) {$\Z$};	
			\draw (12,0)--(14,0);	
			\draw (14,0)--(14,2);	
			\draw (14,2)--(12,2);
			\draw (12,0)--(12,2);
			\draw (12,0)--(14,2);
			\draw (12,2)--(14,0);
			
		\end{tikzpicture}
	\end{center}
	\caption{Examples of graph products.}
\end{figure}
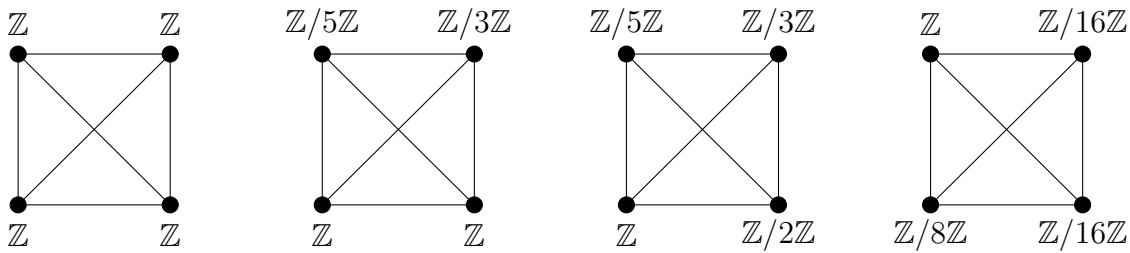

\subsection{The center of the automorphism group of a non-abelian graph product of cyclic groups} Now that we completely understand the abelian case, let us move towards the more general case of a non-abelian graph product. Given a graph product $G_\Gamma$, we have always the epimorphism $\pi=(\pi_1\times\pi_2)\colon\Aut(G_\Gamma)\twoheadrightarrow\Aut(G_\Delta)\times \Aut(G_{\Gamma-\Delta})$ (see \S 3). Since $\pi_1$ and $\pi_2$ are surjective maps and $Z(\Aut(G_{\Gamma-\Delta}))$ is trivial we have  $\pi_1(Z(\Aut(G_\Gamma)))\subseteq Z(\Aut(G_\Delta)))$ and $\pi_2(Z(\Aut(G_\Gamma)))=\left\{id\right\}$. 
\begin{proposition}
\label{Injective}
    Let $\Aut(G_\Gamma)$ be the automorphism group of the graph product $G_\Gamma$ where $V(\Delta)$ and $V(\Gamma)-V(\Delta)$ are both non-empty. Let $k$ denote the maximal order of a vertex in $\Delta$ which is also a power of $2$.
    
    The restriction $\pi_{\mid Z(\Aut(G_\Gamma))}\colon Z(\Aut(G_\Gamma))\to \Aut(G_\Delta)$ is injective if and only if
    \begin{enumerate}
        \item for $v\in V(\Delta)$ we have $ord(v)=\infty$ or $ord(v)=p^n$ for some $n$, $p\neq 2$, or
        \item there are multiple vertices of order $k$ in $\Delta$, or
        \item there is precisely one vertex $v$ of order $k$ in $\Delta$ and either there is no vertex in $\Gamma-\Delta$ such that $\rho_{wv^{ord(v)/2}}$ is well-defined or for all $w_i\in V(\Gamma-\Delta)$ such that $\rho_{w_iv}$ is well-defined which lie in the same orbit of the action of the labeled graph automorphisms on $\Gamma$, also $\rho_{w_iw_j}$ or $\rho_{w_jw_i}$ is well-defined.
    \end{enumerate}
\end{proposition}
\begin{proof}
    The kernel of $\pi$ is given by $T_Z$, recall that these are automorphisms $f$ such that $f(w)=wz_w$ for every vertex $w\in V(\Gamma-\Delta)$ and $f(v)=v$ for every vertex $v\in \Delta$. Let $\iota\colon G_\Gamma\to G_\Gamma$ denote the automorphism induced by $v\mapsto v^{-1}$ for all $v\in V(\Delta)$ and $w\mapsto w$ for all $V(\Gamma-\Delta)$. We can now calculate: $\iota\circ f(w)=wz_w^{-1}$ and $f\circ\iota (w)=wz_w$, hence $ord(z_w)=2$ or $z_w=1$ if $f\in Z(\Aut(G_\Gamma))$.
    
    We first prove the ``if'' statement.
    
    In case (i) there are no elements of order $2$ in $G_\Delta$, thus $z_w=1$ and hence we obtain the injectivity of $\pi_{\mid Z(\Aut(G_\Gamma))}$.
    
    In the setting of case (ii) suppose that there exists $w\in V(\Gamma-\Delta)$ such that $z_w\neq 1$. Then we can write $z_w=v_1^{l_1}v_2^{l_2}...v_o^{l_o}$ where the order of every $v_i$ is a power of $2$ and $l_1,...,l_o\in \Z-\{0\}$. Moreover if there are multiple vertices of the same order in $V(\Delta)$, then either all of these appear as a factor of $z_w$ or none do, since there exists an automorphism of the labeled graph permuting those vertices. Additionally, only a vertex whose order is the maximal power of 2 appearing in $\Delta$ can show up as a factor here, since else a dominated transvection $\rho_{v_ix}$ exists, $x\in V(\Delta)$, with $ord(x)>ord(v_i)$ and then $\rho_{v_ix}\circ f(w)=wz_wx^{ord(v_i/2)}$ and $f\circ \rho_{v_ix}(w)=wz_w$. However there cannot be multiple factors of the same order either, since the dominated transvection $\rho_{v_1v_2}$ does not commute with $f$ either, by essentially the same argument. Therefore we are done with part (ii).
    
    Finally for part (iii) notice that if no dominated transvection $\rho_{wv^{ord(v)/2}}$ exists, then $f(w)=wv^{ord(v)/2}$ will not be well-defined. For the other part of the statement first note that two vertices $w$ and $w'$ in the same orbit need to satisfy $z_w=z_{w'}$, i.e. $f(w')=w'z_w=w'z_{w'}$. The existence of a dominated transvection $\rho_{ww'}$ now implies that $f$ is not central if $z_w\neq 1$, since $f\circ \rho_{ww'}(w)=ww'z_wz_{w'}=ww'$ and $\rho_{ww'}\circ f (w)=ww'z_w$. This proves the ``if'' statement.
    
    For the other direction we essentially mimic the construction of the automorphism $\alpha$ from the clique case. Let $v\in V(\Delta)$ denote the unique vertex in $\Delta$ whose order is the maximal order of $2$ appearing and $w\in V(\Gamma-\Delta)$ denote a vertex such that $\rho_{wv^{ord(v)/2}}$ is well-defined and neither $\rho_{ww'}$ not $\rho_{w'w}$ is well defined for all vertices $w'$ in the orbit of $w$ under the action of the labeled graph automorphisms. Define $\beta\colon G_\Gamma\to G_\Gamma$ as the automorphism induced by $w'\mapsto w'v^{ord(v)/2}$ for all vertices $w'$ in the orbit of $w$ under the action of the labeled graph automorphisms and $x\mapsto x$ for all other vertices. One can check that this is indeed a central element using the generating set from Theorem \ref{GenSet}.
\end{proof}

Let $W_\Gamma$ be a right-angled Coxeter group. It is known that $\Aut(W_\Delta)\cong\GL_n(\Z/2\Z)$ where $n=|V(\Delta)|$. Hence $Z(\Aut(W_\Delta))$ is trivial. Combining this fact with Proposition \ref{Injective}(ii) implies the following result.
\begin{corollary}
Let $W_\Gamma$ be a right-angled Coxeter group. If $|V(\Delta)|\geq 2$, then $Z(\Aut(W_\Gamma))$ is trivial.
\end{corollary}

Now we move on to the case where $\Delta$ has precisely one vertex but there are no further restrictions.

\begin{proposition}
\label{CenterDelta1}
  Let $G_\Gamma$ be a graph product where $\Delta=\left\{v\right\}$ and $V(\Gamma)-V(\Delta)\neq\emptyset$. Let $V(\Gamma)-V(\Delta)=\left\{w_1,\ldots, w_m\right\}$. 
  \begin{enumerate}
      \item If $ord(v)=\infty$, then $Z(\Aut(G_\Gamma))$ is non-trivial if and only if $ord(w_j)<\infty$ for all $j=1,\ldots, m$. In the non-trivial case we have $Z(\Aut(G_\Gamma))=\left\{id, \iota\right\}$.
      \item Assume that $ord(v)=p^k$ and $p\neq 2$.
      \begin{enumerate}
          \item If $ord(w_j)=p_j^{k_j}$ and $p_j\neq p$ for all $j=1,\ldots, m$, then $Z(\Aut(G_\Gamma))\cong (\Z/p^k\Z)^*$.
      \item If there exists $w_i$ with $ord(w_i)=\infty$ or $p^k|ord(w_i)$, then $Z(\Aut(G_\Gamma))$ is trivial.
      \item If there exist vertices $w_1,...,w_m$ such that $ord(w_i)=p^{k_i}$, $k_i<k$, then $Z(\Aut(G_\Gamma))\cong \Z/p^l\Z$ for $l=\min \{p^{k-k_i}| i=1,...,m\}$.
      \end{enumerate}
 
  \item Assume that $ord(v)=2^k$.
  \begin{enumerate}
      \item Assume that $k=1$. Let $l_i$ denote the number of orbits of the action of the graph isometry group on the vertices of order $2^i$. Let $k_i$ denote the number of orbits of vertices of order $2^i$ such that $st(w)\nsubseteq st(w')$ for all vertices $w,w'$ in the orbit. Finally let $m_i$ denote the number of orbits which additionally satisfy $st(x)\nsubseteq st(y)$ for all vertices $x$ in all orbits of order of vertices $<2^i$ and all vertices $y$ in an orbit of vertices of order $2^i.$ Then $Z({\rm Aut}(G_\Gamma)\cong \prod_{i\in\mathbb{N}_{>0}} \left(\Z/2\Z\right)^{m_{i}}$. 
      \item If $ord(v)=2^k$ for some $k>1$, then the center is the direct product of the centers obtained by (ii) (a-c) and (iii) (a).
  \end{enumerate}
   \end{enumerate}
\end{proposition}
\begin{proof}
With regard to (i):\\
If $ord(v)=\infty$ and $ord(w_j)<\infty$ for all $j=1,\ldots, m$, then the special subgroup $G_{\Gamma-\Delta}$ is characteristic by Lemma \ref{InfiniteFiniteCharacteristic}. Thus,  both factors of $G_\Gamma\cong G_\Delta\times G_{\Gamma-\Delta}$ are characteristic and by Lemma \ref{CharacteristicBothFactors} it follows that
$\Aut(G_\Gamma)\cong\Aut(G_\Delta)\times\Aut(G_{\Gamma-\Delta})$.
For the center of $\Aut(G_\Gamma)$ we have: $$Z(\Aut(G_\Gamma))=Z(\Aut(G_\Delta))\times Z(\Aut(G_{\Gamma-\Delta}))=Z(\Aut(G_\Delta))=\left\{id, \iota\right\}.$$
    
If $ord(v)=\infty$ and there exists a vertex $w_i$ with $ord(w_i)=\infty$, then we first show that a central element $f\in Z(\Aut(G_\Gamma))$  maps $v$ to $v$. By Lemma \ref{center} we know that there exists $z_i\in G_\Delta$ such that $f(w_i)=w_iz_i$. Further we also know that $f(v)=v^{\epsilon}$ where $\epsilon =1$ or $-1$. Since the order of $w_i$ is infinite there exists the dominated transvection $\rho_{w_i v}$. Thus
$$f\circ \rho_{w_i v}(w_i)=f(w_iv)=w_iz_iv^{\epsilon}\text{ and }
 \rho_{w_i v}\circ f(w_i)=\rho_{w_iv}(w_iz_i)=w_ivz_i$$    
 
Thus $\epsilon=1$ and $f(v)=v$.  

Furthermore, for a vertex $w_j$ we know by Lemma \ref{center} that there exists $v^k\in G_\Delta=\langle v\rangle$ such that $f(w_j)=w_jv^k$. We have
$$f\circ \iota (w_i)=f(w_i)=w_iv^{k}\text{ and }\iota\circ f(w_i)=\iota(w_iv^k)=w_iv^{-k}$$

Thus $v^k=v^{-k}$ and therefore $k=0$ which shows that $f=id$.

With regard to (ii) (a):\\
By the same argument as in (i) we know that $G_{\Gamma-\Delta}$ is characteristic. Thus
$$Z(\Aut(G_\Gamma))\cong Z(\Aut(G_\Delta))\times Z(\Aut(G_{\Gamma-\Delta}))=Z(\Aut(\Z/p^k\Z))\cong(\Z/p^k\Z)^*.$$

With regard to (ii) (b):\\
With the same strategy as in (i) it is possible to show that for $f\in Z(\Aut(G_\Gamma))$ we have $f(v)=v$.

For $w_i$ there exits $z_i\in \langle v\rangle$ such that $f(w_i)=w_iz_i$. Our goal is to show that $z_i=1$. If $ord(w_i)=p_i^{k_i}$ and $p_i\neq p$, then $z_i=1$. If $ord(w_i)=\infty$ or $ord(w_i)=p^s$, then let $l\in \Aut(G_\Gamma)$ be the local automorphisms of $G_{w_i}$ that maps $w_i$ to $w_i^{-1}$. We have
$$f\circ l(w_i)=f(w_i^{-1})=w_i^{-1}z_i^{-1}\text{ and }
l\circ f(w_i)=l(w_iz_i)=w_i^{-1}z_i.$$
Thus $z_i=z_i^{-1}$ and therefore $z_i^{2}=1$. Furthermore, $p^d=ord(z_i)\mid 2$. By assumption $p\neq 2$ and therefore $ord(z_i)=1$. Which shows again that $f=id$.

With regard to (ii) (c):\\
Let $l_k$ denote the local automorphism of induced by $v\mapsto v^k$ in $\Z/p^k\Z$. Due to the condition on the orders, dominated transvections are only well defined if they have the form $\rho_{wv^{p^{k-i}}}$ where the order of $w$ is $p^i$. Analogously to before a central element in Aut$(G_\Gamma)$ acts as the identity on $G_{\Gamma- \Delta}$. So to compute the center it is enough to compute when $\rho_{wv^{p^{k-i}}}$ and $l_k$ commute. A quick calculation shows that these commute if and only if $wv^{p^{j-i}}=wv^{kp^{j-i}}$ or in other words $p^{j-1}\equiv kp^{j-i}$ mod $p^j$. The minimal non-trivial solution for $k$ is precisely $p^i+1$ and other solutions are obtained by adding $p$ to $k$, hence one dominated transvection of this form reduces the center to $\Z/p^{j-i}\Z=\langle p^i\rangle \leq \Z/p^j\Z$. Note that the generator $p^i$ corresponds to the local automorphism $l_{p^i+1}$. If there are multiple vertices, then since $p^x\mid p^y$ for all $x<y$ it suffices to consider the vertex group with the largest exponent, hence we are done.

With regard to (iii) (a):\\
Let us first consider the right-angled Coxeter group case.\\
In this case we have $Z({\rm Aut}(G_\Delta))=\{1\}$ and hence the center lies completely in $T_Z$. Clearly every automorphism in $T_Z$ commutes with every inner automorphism and every local automorphism as well. Moreover an element $f$ in $T_Z$ has the form $w\mapsto wv$ and $w'\mapsto w'$ for some partition $V(\Gamma-\Delta)=W\cup W'$ and $w\in W, w'\in W'$. If there exists a graph automorphism $\sigma$ which maps some $w\in W$ to a $w'\in W'$ (or the other way around), then $f$ is not central, since $f$ and $\sigma$ do not commute. Moreover if there exists a dominated transvection $\rho_{xy}$ of two vertices $x,y\in W$, then $f$ is also not central, since it does not commute with $\rho_{xy}$. However $f$ commutes with all other dominated transvections and all other graph automorphisms, hence is central. 

Now we consider the general case. Clearly vertex groups of order $p^l$ for $p\neq 2$ do not play a role, since $w\mapsto wv$ is not well-defined if $o(w)=p^l$ and $p\neq 2$. Moreover we can map every vertex $w$ of order $2^n$ to $wv$ and obtain an element of $T_Z$. Within the collection of vertices in $V(\Gamma-\Delta)$ of the same order the same restrictions as above apply about the existence of graph automorphisms and dominated transvections. However we also need to consider dominated transvections between vertices $w$ and $w'$ in $V(\Gamma-\Delta)$ for $ord(w)\neq ord(w')$ if their respective stars satisfy the necessary condition. Let $ord(w)=2^n>2^m=ord(w')$. Suppose $f\in Z({\rm Aut}(G_\Gamma))$ and $f(w)=wv$, $f(w')=w'v$. Then an easy computation shows that $f$ commutes with $\rho_{ww'^{2^{n-m}}}$ but it does not commute with $\rho_{w'w}$ leading to the additional condition on the stars in the Lemma. However if such a dominated transvection does not exist, then $f$ is central as in the right-angled Coxeter group case.

With regard to (iii) (b):\\
Since the order of $v$ is $2^k$, we now obtain local automorphisms of $G_\Delta$. These are of the form $v\mapsto v^n$ for some $n$. However since the order of $v$ is $2^k$, all local automorphisms are of the form $(v\mapsto v^{-1})^m$ for some exponent $m\in \mathbb{N}$. An easy calculation shows that a dominated transvection $\rho_{wv^l}$ commutes with the local automorphism $i: v\mapsto v^{-1}$ if and only if $o(v^l)\leq 2$. This also shows that such an automorphism commutes with every local automorphism of $\Z/2^k\Z$. Hence applying the exact same arguments as in part (iii) (a) we obtain a factor $\left(\Z/2\Z\right)^{m_{i}}$ in the center. To answer the question whether some local automorphism can be central we refer to part (ii), since these arguments show exactly in which cases a local automorphism can be central. The assumption that $p\neq 2$ was merely used to show that $Z(\Aut(G_\Gamma))\cap T_Z$ is trivial, the other arguments work even if $p=2$. Hence we obtain the desired statement.
\end{proof}
\begin{remark}
If the graph $\Delta$ has more vertices than just one, similar methods can be applied to calculate the center in any given case. More precisely one can determine the center of $\Aut(G_\Delta)$ with Proposition \ref{centerDelta} and then follow the arguments above to determine the center. 
\end{remark}

We summarize the results regarding the center of the automorphism group of a right-angled Coxeter group in the following corollary.

\begin{corollary}
\label{CenterRACG}
	Let $W_\Gamma$ be a right-angled Coxeter group. Let $\Delta$ be the induced subgraph of $\Gamma$ generated by the vertices $v\in V(\Gamma)$ such that $st(v)=V(\Gamma)$. Further, we denote the vertices of $\Delta$ by $V(\Delta)=\left\{v_1,\ldots, v_n\right\}$ and  $V(\Gamma)-V(\Delta)=\left\{w_1,\ldots, w_m\right\}$.  
	 
	 The center $Z(\Aut(W_\Gamma))$ is non-trivial if and only if $n=1$, $m\geq 1$ and there exists a vertex $w_j\in V(\Gamma)-V(\Delta)$ such that $st(w_j)\nsubseteq st(w_i)$ and $st(w_i)\nsubseteq st(w_j)$ for all $i\in\left\{1,\ldots, m\right\}, i\neq j$. Moreover, let $\Omega$ be a subset of $V(\Gamma)-V(\Delta)$ defined as follows: 
	$$\Omega:=\left\{w_j\mid st(w_j)\nsubseteq st(w_i)\text{ and } st(w_i)\nsubseteq st(w_j) \text{ for all }i\in\left\{1,\ldots, m\right\}, i\neq j\right\}.$$ 
	Then $\Omega$ is preserved under the action of ${\rm Isom}(\Gamma-\Delta)$ and $Z(\Aut(W_\Gamma))\cong (\Z/2\Z)^l$ where $l$ is equal to the cardinality of the set of orbits under this action.
\end{corollary}

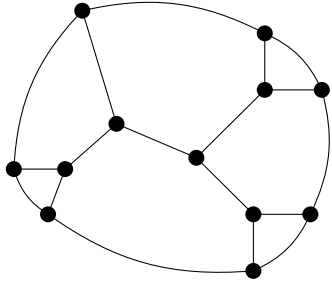
\begin{figure}[h]
	\begin{center}
		\begin{tikzpicture}[scale=1.5]
			\coordinate (A) at (-1,1.3);
			\coordinate (B) at (0.6,1.1);
			\coordinate (C) at (1.1,0.6);
			\coordinate (D) at (-1.15,-0.1);
			\coordinate (E) at (-1.3,-0.5);
			\coordinate (F) at (1,-0.5);
			\coordinate (G) at (0.5,-1);
			\coordinate (H) at (-1.6,-0.1);
			\coordinate (I) at (-0.7,0.3);
			\coordinate (J) at (0.6,0.6);
			\coordinate (K) at (0.5,-0.5);
			\coordinate (L) at (0,0);
			\draw[bend right=20] (A) to (H);
			\draw[bend right=20] (H) to (E);
			\draw[bend right=20] (E) to (G);
			\draw[bend right=20] (G) to (F);
			\draw[bend right=20] (F) to (C);
			\draw[bend right=20] (C) to (B);
			\draw[bend right=20] (B) to (A);
			\draw (A) to (I);
			\draw (I) to (D);
			\draw (D) to (H);
			\draw (D) to (E);
			\draw (I) to (L);
			\draw (L) to (K);
			\draw (K) to (F);
			\draw (K) to (G);
			\draw (J) to (B);
			\draw (J) to (C);
			\draw (J) to (L);
			\fill[color=black] (A) circle (2pt);
			\fill[color=black] (B) circle (2pt);
			\fill[color=black] (C) circle (2pt);
			\fill[color=black] (D) circle (2pt);
			\fill[color=black] (E) circle (2pt);
			\fill[color=black] (F) circle (2pt);
			\fill[color=black] (G) circle (2pt);
			\fill[color=black] (H) circle (2pt);
			\fill[color=black] (I) circle (2pt);
			\fill[color=black] (J) circle (2pt);
			\fill[color=black] (K) circle (2pt);
			\fill[color=black] (L) circle (2pt);
			
		\end{tikzpicture}
		\caption{Frucht graph.}
	\end{center}
\end{figure}
An example of a graph satisfying the condition on the stars is the Frucht graph introduced in \cite{Frucht}, see Figure 7. 

Hence the automorphism group of the right-angled Coxeter group associated to the join of the Frucht graph with a single vertex has non-trivial center. More precisely the isometry group of this graph is trivial, thus the center of the automorphism group of this right-angled Coxeter group is in fact isomorphic to $(\Z/2\Z)^{12}$.

\section{Applications}
In this chapter we provide two applications of our results. The first application is to automatic continuity and the second one is to the stable rank of the reduced group $C^\ast$-algebra. We start this section by discussing the question of automatic continuity and giving some examples and non-examples. We then prove Corollary I, which we show to be ``optimal'' in the case of right-angled Artin groups. Afterwards we move on to the second application and prove Corollary J.

Given a map $\varphi\colon L\to G$ between two topological groups $L$ and $G$, the \emph{automatic continuity problem} is the following: assuming $\varphi$ is a
group homomorphism on the level of groups, find conditions on $L$, $G$ or $\varphi$ which imply
that $\varphi$ is continuous.

Here we focus on the case where $L$ is an arbitrary locally compact Hausdorff group and $G$ has the discrete topology. By definition, a discrete group $G$ is called \emph{lcH-slender} if every algebraic homomorphism $\varphi\colon L\to G$ where $L$ is a locally compact Hausdorff group is continuous. Important examples of lcH-slender groups are right-angled Artin groups \cite{CorsonKazachkov}, \cite{KramerVarghese}, torsion free hyperbolic groups \cite{ConnerCorson}, torsion free CAT(0) groups \cite{CorsonVarghese} and torsion free Helly groups \cite{KeppelerMoellerVarghese}. An algebraic characterization of lcH-slender abelian groups was given in \cite{Corson} and was generalized for  lcH-slender groups in \cite{CorsonVarghese}.

Torsionfreeness is a necessary condition for lcH-slenderness. This follows from the fact, that if $G$ has a torsion element, then $G$ has an element of order $p$ where $p$ is a prime number and we can always define a discontinuous group homomorphism $\varphi\colon\prod_{\mathbb{N}}\Z/p\Z\twoheadrightarrow \Z/p\Z$ using a vector space argument (see the proof of \cite[Thm. C]{Corson}). Thus, if  $G$ is infinite, then this discontinuous group homomorphism can not interact with the entire structure of the group, since its image is finite.  This leads to the following question.
\begin{question}
Under which algebraic conditions on the discrete group $G$ is every algebraic epimorphism $\varphi\colon L\twoheadrightarrow G$ continuous?
\end{question}

The next lemma shows, that if the center of $G$ is not well-behaved in the sense of automatic continuity, then there exists a discontinuous algebraic epimorphism from a locally compact Hausdorff group to $G$.
\begin{lemma}
	\label{discCenter}
	Let $G$ be a group. Assume that there exists a discontinuous group homomorphism $\varphi_1\colon L\to Z(G)$ where $L$ is a locally compact Hausdorff group and $Z(G)$ is the center of $G$. Then the map $\varphi\colon L\times G\to G$ defined as $\varphi(l,g)=\varphi_1(l)\cdot g$ is a discontinuous surjective group homomorphism where $G$ has the discrete topology and $L\times G$ has the product topology.	
\end{lemma}
\begin{proof}
	The map $\varphi$ is a group homomorphism, since $\varphi_1(L)$ is contained in the center of $G$. Furthermore, we have $\varphi(\left\{1\right\}\times G)=G$, hence $\varphi$ is surjective.
	We have $\varphi_{|L}=\varphi_1$ and this map is by assumption discontinuous, therefore $\varphi$ is also discontinuous.
\end{proof}

\begin{corollary}
\label{DisCenter}
	Let $G$ be a group. If the center of $G$ has a non-trivial finite order element, or contains $\mathbb{Q}$ as a subgroup or the $p$-adic integers $\Z_p$ for a prime $p$, then there exists a discontinuous algebraic epimorphism from a locally compact Hausdorff group to $G$.
\end{corollary}
\begin{proof}
    It was proven in \cite{CorsonVarghese} that a group $H$ is lcH-slender if and only if $H$ is torsion free, contains no $\mathbb{Q}$ and no $\Z_p$ as a subgroup. Hence there exists a discontinuous algebraic homomorphism from a locally compact Hausdorff group to $Z(G)$. Thus, by Lemma \ref{discCenter} there exists a discontinuous epimorphism from a locally compact Hausdorff group to $G$.
\end{proof}
\begin{corollary}
Let $G_\Gamma$ denote a graph product of finitely generated abelian groups. If the center $Z(\Aut(G_\Gamma))$ is non-trivial, then there exist a locally compact group $L$ and a discontinuous surjective homomorphism $\varphi\colon L\twoheadrightarrow \Aut(G_\Gamma)$. 
\end{corollary}
\begin{proof}
    This follows from the previous corollary, since the center of the automorphism group is finite by Lemma \ref{CenterFiniteDelta}.
\end{proof}

The question above regarding automatic continuity of epimorphisms from a locally compact Hausdorff group to a given discrete group was partially answered in \cite{KeppelerMoellerVarghese}. Here we recall a weaker version of this result which is suitable for our investigations. 
\begin{theorem}(\cite{KeppelerMoellerVarghese})
\label{autTheorem}
Let $\varphi\colon L\twoheadrightarrow G$ be an epimorphism from a locally compact Hausdorff group $L$ to a discrete group $G$. If every torsion subgroup of $G$ is finite, $G$ does not contain $\mathbb{Q}$ or the $p$-adic integers $\Z_p$ as a subgroup and $G$ does not have non-trivial finite normal subgroups, then $\varphi$ is continuous.
\end{theorem}
It is known that for a given finite group $G$ there exists always a discontinuous surjective group homomorphism from the compact group $\prod_{\mathbb{N}}G$ into $G$, see \cite[Example 4.2.12]{Ribes}.
In general, it may be difficult to find a discontinuous surjective group homomorphism. Nevertheless we conjecture the following result.
\begin{conjecture}
Let $G$ be a group. Assume that every torsion subgroup in $G$ is finite, $G$ does not contain $\mathbb{Q}$ or the $p$-adic integers $\Z_p$ as a subgroup for any prime $p$. The following statements are equivalent:
\begin{enumerate}
    \item Every epimorphism from a locally compact Hausdorff group $L$ to $G$ is continuous.
    \item The group $G$ does not have non-trivial finite normal subgroups.
\end{enumerate}
\end{conjecture}
Corollary \ref{DisCenter} implies that if $G$ has a non-trivial finite normal subgroup $N$ such that $N\subseteq Z(G)$, then there exists a discontinuous epimorphism from a locally compact Hausdorff group to $G$. Further, if $G$ has a non-trivial finite normal subgroup $N$ such that $G\cong N\times M$, then there exists a discontinuous epimorphism $\varphi\colon\prod_{\mathbb{N}}N\times M\twoheadrightarrow G$. Moreover, the centralizer of a non-trivial finite normal subgroup $N\trianglelefteq G$, defined by $Z_G(N)=\left\{g\in G\mid gn=ng \text{ for all }n\in N\right\}$ has finite index in $G$, therefore the same construction as in Lemma \ref{discCenter} shows that there exists a surjective, discontinuous group homomorphism $\varphi\colon \prod_\mathbb{N} N\times Z_G(N)\twoheadrightarrow \langle N, Z_G(N)\rangle$ where the subgroup $\langle Z_G(N), N\rangle$ has finite index in $G$. Thus, if $G$ has a non-trivial finite normal subgroup, then there exists a surjective, discontinuous group homomorphism from a locally compact Hausdorff group to a finite index subgroup of $G$.

Recall, a group $G$ is by definition \emph{complete} if $G$ is centerless and every automorphism of $G$ is inner. For example, the group ${\rm Sym}(n)$ is complete for $n\neq 2,6$. The notion of completeness goes back to  H\"older \cite{Hoelder}, where he studied decompositions of a group $G$. More precisely, given a group $G$ and a fixed normal subgroup $N$ of $G$ one can ask the question how the group $N$ is involved in the decompositions of a group $G$ in smaller pieces where one piece is equal to the group $N$. H\"older proved that if $N$ is  complete, then any short exact sequence 
$$\left\{1\right\}\to N\to G\to M\to\left\{1\right\}$$
splits and $G$ is isomorphic to the direct product $N\times M$. Hence, if $G$ has a non-trivial complete finite normal subgroup, then there exists a discontinuous epimorphism from a locally compact Hausdorff group to $G$. 

Given a graph product $G_\Gamma$ of cyclic groups, then Corollary \ref{TorsionFinite} implies that torsion subgroups in $\Aut(G_\Gamma)$ are finite. Moreover, by Lemma \ref{noQ} we also know that $\mathbb{Q}$ is not a subgroup of $\Aut(G_\Gamma)$. Additionally, $\Aut(G_\Gamma)$ is finitely generated by Theorem \ref{GenSet}, hence countable and therefore can not contain the uncountable group $\Z_p$ as a subgroup. Hence, Theorem \ref{autTheorem} immediately implies the following result.
\begin{corollary}
\label{automaticAutG}
Let $\varphi\colon L\twoheadrightarrow{\rm Aut}(G_\Gamma)$ be an abstract group epimorphism from a locally compact Hausdorff group $L$ to the automorphism group of a graph product of cyclic groups $G_\Gamma$. If $\Aut(G_\Gamma)$ does not have non-trivial finite normal subgroups, then $\varphi$ is continuous. 
\end{corollary}

In the special case where $G_\Gamma$ is a right-angled Artin group we obtain a precise characterization of automatic continuity.
\begin{corollary}
\label{automaticRAAG}
	Let $\varphi\colon L\twoheadrightarrow{\rm Aut}(A_\Gamma)$ be an abstract group epimorphism from a locally compact Hausdorff group $L$ into the automorphism group of a right angled Artin group $A_\Gamma$. 
	\begin{enumerate}
		\item If $\Gamma$ is not a clique, then $\varphi$ is continuous.
		\item If $\Gamma$ is a clique, then 
		$$\pi\circ\varphi\colon L\twoheadrightarrow{\rm Aut}(A_\Gamma)\cong {\rm GL}_n(\mathbb{Z})\twoheadrightarrow{\rm PGL}_n(\mathbb{Z})$$
		is continuous, where $\pi$ is the canonical projection.
	\end{enumerate}
Moreover, there exists a discontinuous epimorphism $\varphi\colon\prod_{\mathbb{N}}\Z/2\Z\times {\rm GL}_n(\Z)\twoheadrightarrow{\rm GL}_n(\Z).$
\end{corollary}

\begin{proof}
    The first result follows immediately from Corollary \ref{automaticAutG} and Corollary C. Further, the group ${\rm PGL}_n(\Z)$ does not have finite normal subgroups and also satisfies the properties of Theorem \ref{autTheorem}, therefore every algebraic group epimorphism from a locally compact Hausdorff group $L$ to ${\rm PGL}_n(\Z)$ is continuous. 
    
    Now, we give a discontinuous epimorphism from a locally compact Hausdorff group to $\GL_n(\Z)$.
	The center of ${\rm GL}_n(\Z)$ is equal to $\left\{I, -I\right\}$. Thus there exists a discontinuous group homomorphism 
	$\varphi_1\colon\prod_{\mathbb{N}}\Z/2\Z\to \Z/2\Z\to \left\{I,-I\right\}$, for explicit definition of this group homomorphism see \cite{MoellerVarghese}.
	By Lemma \ref{discCenter}, the group homomorphism $\varphi\colon\prod_{\mathbb{N}}\Z/2\Z\times {\rm GL}_n(\Z)\twoheadrightarrow{\rm GL}_n(\Z)$ where $\varphi(l,g)=\varphi_1(l)\cdot g$ is a discontinuous surjective group homomorphism.
\end{proof}

For the second part of this chapter we deal with reduced group $C^\ast$-algebras. Associated to a group $G$ there exists an associated reduced group $C^\ast$-algebra which we denote by $C_r^\ast(G)$. For the precise definition and more information see for example \cite{GerasimovaOsin}. We say $C_r^\ast(G)$ \textit{has stable rank} $1$ if the group of invertible elements is dense in $C_r^\ast(G)$. Recall that a group $G$ is called \textit{acylindrically hyperbolic} if it acts acylindrically on a Gromov-hyperbolic space. A good reference for reduced group $C^\ast$-algebras of acylindrically hyperbolic groups is \cite{GerasimovaOsin} and a good reference for acylindrical hyperbolicity of automorphism groups of graph products is \cite{Genevois}. We don't want to go into more detail here and instead move on to the proof of Corollary J. 
\begin{proof}[Proof of Corollary J]
    First we invoke \cite[Thm. 1.1]{Genevois} to see that $\Aut(G_\Gamma)$ is an acylindrically hyperbolic group. Then \cite[Thm 1.1]{GerasimovaOsin} implies that the stable rank of $C^\ast_r(\Aut(G_\Gamma))$ is $1$ if the amenable radical is trivial, which is in particular the case if there are no non-trivial finite normal subgroups. Proposition E and Theorem G imply that, in the setting of Corollary J, $\Aut(G_\Gamma)$ does not have non-trivial finite normal subgroup, which proves Corollary J.
\end{proof}
Contrary to the automatic continuity situation we do not know whether this result is ``optimal'' even for right-angled Artin groups. In fact we believe even the ``easiest'' case to be open, which is the following question.
\begin{question}
Let $n\geq 2$, is the stable rank of $C_r^\ast(\GL_n(\Z))$ equal to 1?
\end{question}

\end{document}